\documentclass[12pt, reqno]{amsart}
\usepackage{txfonts}      %{article} was 12pt latex e
\usepackage{amssymb}
\usepackage{eucal}
\usepackage{amsmath}
\usepackage{amsthm}
\usepackage{amscd}
\usepackage[dvips]{color}
\usepackage{multicol}
\usepackage[all]{xy}           %xypic macro for latex2.09
\usepackage{graphicx}
\usepackage{color}
\usepackage{colordvi}
\usepackage{xspace}
\usepackage{tikz}

\usepackage{ifpdf}
\ifpdf
 \usepackage[colorlinks,final,backref=page,hyperindex]{hyperref}
\else
 \usepackage[colorlinks,final,backref=page,hyperindex,hypertex]{hyperref}
\fi

\begin{document}
\newcommand {\emptycomment}[1]{} %to remove paragraphs

\newcommand{\nc}{\newcommand}
\newcommand{\delete}[1]{}
\nc{\mfootnote}[1]{\footnote{#1}} % Use this to show footnotes
\nc{\todo}[1]{\tred{To do:} #1}

%\delete{
\nc{\mlabel}[1]{\label{#1}}  % Use this to suppress names
\nc{\mcite}[1]{\cite{#1}}  % Use this to suppress names
\nc{\mref}[1]{\ref{#1}}  % Use this to suppress names
\nc{\meqref}[1]{\eqref{#1}} % Use this to suppress names
\nc{\mbibitem}[1]{\bibitem{#1}} % Use this to show number
%}

\delete{
\nc{\mlabel}[1]{\label{#1}  % Use the next two lines to show names
{\hfill \hspace{1cm}{\bf{{\ }\hfill(#1)}}}}
\nc{\mcite}[1]{\cite{#1}{{\bf{{\ }(#1)}}}}  % Use this lines to show names
\nc{\mref}[1]{\ref{#1}{{\bf{{\ }(#1)}}}}  % Use this lines to show names
\nc{\meqref}[1]{\eqref{#1}{{\bf{{\ }(#1)}}}} % Use this lines to show names
\nc{\mbibitem}[1]{\bibitem[\bf #1]{#1}} % Use this to show name
}

%%%%%%%%%%%%%%%%%%%%%%%% Statements
\newtheorem{thm}{Theorem}[section]
\newtheorem{lem}[thm]{Lemma}
\newtheorem{cor}[thm]{Corollary}
\newtheorem{pro}[thm]{Proposition}
\theoremstyle{definition}
\newtheorem{defi}[thm]{Definition}
\newtheorem{ex}[thm]{Example}
\newtheorem{rmk}[thm]{Remark}
\newtheorem{pdef}[thm]{Proposition-Definition}
\newtheorem{condition}[thm]{Condition}

\renewcommand{\labelenumi}{{\rm(\alph{enumi})}}
\renewcommand{\theenumi}{\alph{enumi}}
\renewcommand{\labelenumii}{{\rm(\roman{enumii})}}
\renewcommand{\theenumii}{\roman{enumii}}

\nc{\tred}[1]{\textcolor{red}{#1}}
\nc{\tblue}[1]{\textcolor{blue}{#1}}
\nc{\tgreen}[1]{\textcolor{green}{#1}}
\nc{\tpurple}[1]{\textcolor{purple}{#1}}
\nc{\btred}[1]{\textcolor{red}{\bf #1}}
\nc{\btblue}[1]{\textcolor{blue}{\bf #1}}
\nc{\btgreen}[1]{\textcolor{green}{\bf #1}}
\nc{\btpurple}[1]{\textcolor{purple}{\bf #1}}

\nc{\ts}[1]{\textcolor{purple}{Tianshui:#1}}
\nc{\cm}[1]{\textcolor{red}{Chengming:#1}}
\nc{\li}[1]{\textcolor{red}{#1}}
\nc{\lir}[1]{\textcolor{blue}{Li:#1}}

%%%%%%%%%%%%%% Matrix symbols.

\nc{\twovec}[2]{\left(\begin{array}{c} #1 \\ #2\end{array} \right )}
\nc{\threevec}[3]{\left(\begin{array}{c} #1 \\ #2 \\ #3 \end{array}\right )}
\nc{\twomatrix}[4]{\left(\begin{array}{cc} #1 & #2\\ #3 & #4 \end{array} \right)}
\nc{\threematrix}[9]{{\left(\begin{matrix} #1 & #2 & #3\\ #4 & #5 & #6 \\ #7 & #8 & #9 \end{matrix} \right)}}
\nc{\twodet}[4]{\left|\begin{array}{cc} #1 & #2\\ #3 & #4 \end{array} \right|}

\nc{\rk}{\mathrm{r}}
\newcommand{\g}{\mathfrak g}
\newcommand{\h}{\mathfrak h}
\newcommand{\pf}{\noindent{$Proof$.}\ }
\newcommand{\frkg}{\mathfrak g}
\newcommand{\frkh}{\mathfrak h}
\newcommand{\Id}{\rm{Id}}
\newcommand{\gl}{\mathfrak {gl}}
\newcommand{\ad}{\mathrm{ad}}
\newcommand{\add}{\frka\frkd}
\newcommand{\frka}{\mathfrak a}
\newcommand{\frkb}{\mathfrak b}
\newcommand{\frkc}{\mathfrak c}
\newcommand{\frkd}{\mathfrak d}
\newcommand {\comment}[1]{{\marginpar{*}\scriptsize\textbf{Comments:} #1}}
%%%%%%%%%%%%%%%%%%%%%%% symbols

\nc{\tforall}{\text{ for all }}

\nc{\svec}[2]{{\tiny\left(\begin{matrix}#1\\
#2\end{matrix}\right)\,}}  % column vector
\nc{\ssvec}[2]{{\tiny\left(\begin{matrix}#1\\
#2\end{matrix}\right)\,}} % subscript column vector

\nc{\typeI}{local cocycle $3$-Lie bialgebra\xspace}
\nc{\typeIs}{local cocycle $3$-Lie bialgebras\xspace}
\nc{\typeII}{double construction $3$-Lie bialgebra\xspace}
\nc{\typeIIs}{double construction $3$-Lie bialgebras\xspace}

\nc{\bia}{{$\mathcal{P}$-bimodule ${\bf k}$-algebra}\xspace}
\nc{\bias}{{$\mathcal{P}$-bimodule ${\bf k}$-algebras}\xspace}

\nc{\rmi}{{\mathrm{I}}}
\nc{\rmii}{{\mathrm{II}}}
\nc{\rmiii}{{\mathrm{III}}}
\nc{\pr}{{\mathrm{pr}}}
\newcommand{\huaA}{\mathcal{A}}

\nc{\OT}{constant $\theta$-}
\nc{\T}{$\theta$-}
\nc{\IT}{inverse $\theta$-}

%new new commands

\nc{\asi}{ASI\xspace}
\nc{\qadm}{$Q$-admissible\xspace}
\nc{\aybe}{AYBE\xspace}
\nc{\admset}{\{\pm x\}\cup (-x+K^{\times}) \cup K^{\times} x^{-1}}

\nc{\dualrep}{gives a dual representation\xspace}
\nc{\admt}{admissible to\xspace}

\nc{\opa}{\cdot_A}
\nc{\opb}{\cdot_B}

\nc{\post}{positive type\xspace}
\nc{\negt}{negative type\xspace}
\nc{\invt}{inverse type\xspace}

\nc{\pll}{\beta}
\nc{\plc}{\epsilon}

\nc{\ass}{{\mathit{Ass}}}
\nc{\lie}{{\mathit{Lie}}}
\nc{\comm}{{\mathit{Comm}}}
\nc{\dend}{{\mathit{Dend}}}
\nc{\zinb}{{\mathit{Zinb}}}
\nc{\tdend}{{\mathit{TDend}}}
\nc{\prelie}{{\mathit{preLie}}}
\nc{\postlie}{{\mathit{PostLie}}}
\nc{\quado}{{\mathit{Quad}}}
\nc{\octo}{{\mathit{Octo}}}
\nc{\ldend}{{\mathit{ldend}}}
\nc{\lquad}{{\mathit{LQuad}}}

 \nc{\adec}{\check{;}} \nc{\aop}{\alpha}
\nc{\dftimes}{\widetilde{\otimes}} \nc{\dfl}{\succ} \nc{\dfr}{\prec}
\nc{\dfc}{\circ} \nc{\dfb}{\bullet} \nc{\dft}{\star}
\nc{\dfcf}{{\mathbf k}} \nc{\apr}{\ast} \nc{\spr}{\cdot}
\nc{\twopr}{\circ} \nc{\tspr}{\star} \nc{\sempr}{\ast}
\nc{\disp}[1]{\displaystyle{#1}}
\nc{\bin}[2]{ (_{\stackrel{\scs{#1}}{\scs{#2}}})}  %binomial coeff
\nc{\binc}[2]{ \left (\!\! \begin{array}{c} \scs{#1}\\
    \scs{#2} \end{array}\!\! \right )}  %binomial coeff
\nc{\bincc}[2]{  \left ( {\scs{#1} \atop
    \vspace{-.5cm}\scs{#2}} \right )}  %binomial coeff
\nc{\sarray}[2]{\begin{array}{c}#1 \vspace{.1cm}\\ \hline
    \vspace{-.35cm} \\ #2 \end{array}}
\nc{\bs}{\bar{S}} \nc{\dcup}{\stackrel{\bullet}{\cup}}
\nc{\dbigcup}{\stackrel{\bullet}{\bigcup}} \nc{\etree}{\big |}
\nc{\la}{\longrightarrow} \nc{\fe}{\'{e}} \nc{\rar}{\rightarrow}
\nc{\dar}{\downarrow} \nc{\dap}[1]{\downarrow
\rlap{$\scriptstyle{#1}$}} \nc{\uap}[1]{\uparrow
\rlap{$\scriptstyle{#1}$}} \nc{\defeq}{\stackrel{\rm def}{=}}
\nc{\dis}[1]{\displaystyle{#1}} \nc{\dotcup}{\,
\displaystyle{\bigcup^\bullet}\ } \nc{\sdotcup}{\tiny{
\displaystyle{\bigcup^\bullet}\ }} \nc{\hcm}{\ \hat{,}\ }
\nc{\hcirc}{\hat{\circ}} \nc{\hts}{\hat{\shpr}}
\nc{\lts}{\stackrel{\leftarrow}{\shpr}}
\nc{\rts}{\stackrel{\rightarrow}{\shpr}} \nc{\lleft}{[}
\nc{\lright}{]} \nc{\uni}[1]{\tilde{#1}} \nc{\wor}[1]{\check{#1}}
\nc{\free}[1]{\bar{#1}} \nc{\den}[1]{\check{#1}} \nc{\lrpa}{\wr}
\nc{\curlyl}{\left \{ \begin{array}{c} {} \\ {} \end{array}
    \right .  \!\!\!\!\!\!\!}
\nc{\curlyr}{ \!\!\!\!\!\!\!
    \left . \begin{array}{c} {} \\ {} \end{array}
    \right \} }
\nc{\leaf}{\ell}       % number of leafs
\nc{\longmid}{\left | \begin{array}{c} {} \\ {} \end{array}
    \right . \!\!\!\!\!\!\!}
\nc{\ot}{\otimes} \nc{\sot}{{\scriptstyle{\ot}}}
\nc{\otm}{\overline{\ot}}
\nc{\ora}[1]{\stackrel{#1}{\rar}}
\nc{\ola}[1]{\stackrel{#1}{\la}}%${\Bbb Z}$
\nc{\pltree}{\calt^\pl}
\nc{\epltree}{\calt^{\pl,\NC}}
\nc{\rbpltree}{\calt^r}
\nc{\scs}[1]{\scriptstyle{#1}} \nc{\mrm}[1]{{\rm #1}}
\nc{\dirlim}{\displaystyle{\lim_{\longrightarrow}}\,}
\nc{\invlim}{\displaystyle{\lim_{\longleftarrow}}\,}
\nc{\mvp}{\vspace{0.5cm}} \nc{\svp}{\vspace{2cm}}
\nc{\vp}{\vspace{8cm}} \nc{\proofbegin}{\noindent{\bf Proof: }}
%\nc{\proofbegin}{\begin{proof}} % AMS command
\nc{\proofend}{$\blacksquare$ \vspace{0.5cm}}
%\nc{\proofend}{\end{proof}} %AMS command
\nc{\freerbpl}{{F^{\mathrm RBPL}}}
\nc{\sha}{{\mbox{\cyr X}}}  %used to be \cyr
\nc{\ncsha}{{\mbox{\cyr X}^{\mathrm NC}}} \nc{\ncshao}{{\mbox{\cyr
X}^{\mathrm NC,\,0}}}
\nc{\shpr}{\diamond}    %Shuffle product
\nc{\shprm}{\overline{\diamond}}    %Shuffle product
\nc{\shpro}{\diamond^0}    %Shuffle product
\nc{\shprr}{\diamond^r}     %product on controlled trees
\nc{\shpra}{\overline{\diamond}^r}
\nc{\shpru}{\check{\diamond}} \nc{\catpr}{\diamond_l}
\nc{\rcatpr}{\diamond_r} \nc{\lapr}{\diamond_a}
\nc{\sqcupm}{\ot}
\nc{\lepr}{\diamond_e} \nc{\vep}{\varepsilon} \nc{\labs}{\mid\!}
\nc{\rabs}{\!\mid} \nc{\hsha}{\widehat{\sha}}
\nc{\lsha}{\stackrel{\leftarrow}{\sha}}
\nc{\rsha}{\stackrel{\rightarrow}{\sha}} \nc{\lc}{\lfloor}
\nc{\rc}{\rfloor}
\nc{\tpr}{\sqcup}
\nc{\nctpr}{\vee}
\nc{\plpr}{\star}
\nc{\rbplpr}{\bar{\plpr}}
\nc{\sqmon}[1]{\langle #1\rangle}
\nc{\forest}{\calf}
\nc{\altx}{\Lambda_X} \nc{\vecT}{\vec{T}} \nc{\onetree}{\bullet}
\nc{\Ao}{\check{A}}
\nc{\seta}{\underline{\Ao}}
\nc{\deltaa}{\overline{\delta}}
\nc{\trho}{\tilde{\rho}}

\nc{\rpr}{\circ}
%\nc{\apr}{\cdot}
\nc{\dpr}{{\tiny\diamond}}
\nc{\rprpm}{{\rpr}}

%%%%%%%%%%%%%%%%%%%%% roman fonts, in alphabetic order
\nc{\mmbox}[1]{\mbox{\ #1\ }} \nc{\ann}{\mrm{ann}}
\nc{\Aut}{\mrm{Aut}} \nc{\can}{\mrm{can}}
\nc{\twoalg}{{two-sided algebra}\xspace}
\nc{\colim}{\mrm{colim}}
\nc{\Cont}{\mrm{Cont}} \nc{\rchar}{\mrm{char}}
\nc{\cok}{\mrm{coker}} \nc{\dtf}{{R-{\rm tf}}} \nc{\dtor}{{R-{\rm
tor}}}
\renewcommand{\det}{\mrm{det}}
\nc{\depth}{{\mrm d}}
\nc{\Div}{{\mrm Div}} \nc{\End}{\mrm{End}} \nc{\Ext}{\mrm{Ext}}
\nc{\Fil}{\mrm{Fil}} \nc{\Frob}{\mrm{Frob}} \nc{\Gal}{\mrm{Gal}}
\nc{\GL}{\mrm{GL}} \nc{\Hom}{\mrm{Hom}} \nc{\hsr}{\mrm{H}}
\nc{\hpol}{\mrm{HP}} \nc{\id}{\mrm{id}} \nc{\im}{\mrm{im}}
\nc{\incl}{\mrm{incl}} \nc{\length}{\mrm{length}}
\nc{\LR}{\mrm{LR}} \nc{\mchar}{\rm char} \nc{\NC}{\mrm{NC}}
\nc{\mpart}{\mrm{part}} \nc{\pl}{\mrm{PL}}
\nc{\ql}{{\QQ_\ell}} \nc{\qp}{{\QQ_p}}
\nc{\rank}{\mrm{rank}} \nc{\rba}{\rm{RBA }} \nc{\rbas}{\rm{RBAs }}
\nc{\rbpl}{\mrm{RBPL}}
\nc{\rbw}{\rm{RBW }} \nc{\rbws}{\rm{RBWs }} \nc{\rcot}{\mrm{cot}}
\nc{\rest}{\rm{controlled}\xspace}
\nc{\rdef}{\mrm{def}} \nc{\rdiv}{{\rm div}} \nc{\rtf}{{\rm tf}}
\nc{\rtor}{{\rm tor}} \nc{\res}{\mrm{res}} \nc{\SL}{\mrm{SL}}
\nc{\Spec}{\mrm{Spec}} \nc{\tor}{\mrm{tor}} \nc{\Tr}{\mrm{Tr}}
\nc{\mtr}{\mrm{sk}}

%%%%%%%%%%%%%%%%%% bold face
\nc{\ab}{\mathbf{Ab}} \nc{\Alg}{\mathbf{Alg}}
\nc{\Algo}{\mathbf{Alg}^0} \nc{\Bax}{\mathbf{Bax}}
\nc{\Baxo}{\mathbf{Bax}^0} \nc{\RB}{\mathbf{RB}}
\nc{\RBo}{\mathbf{RB}^0} \nc{\BRB}{\mathbf{RB}}
\nc{\Dend}{\mathbf{DD}} \nc{\bfk}{{\bf k}} \nc{\bfone}{{\bf 1}}
\nc{\base}[1]{{a_{#1}}} \nc{\detail}{\marginpar{\bf More detail}
    \noindent{\bf Need more detail!}
    \svp}
\nc{\Diff}{\mathbf{Diff}} \nc{\gap}{\marginpar{\bf
Incomplete}\noindent{\bf Incomplete!!}
    \svp}
\nc{\FMod}{\mathbf{FMod}} \nc{\mset}{\mathbf{MSet}}
\nc{\rb}{\mathrm{RB}} \nc{\Int}{\mathbf{Int}}
\nc{\Mon}{\mathbf{Mon}}
%\nc{\remark}{\noindent{\bf Remark: }}
\nc{\remarks}{\noindent{\bf Remarks: }}
\nc{\OS}{\mathbf{OS}} %free operated semigroup
\nc{\Rep}{\mathbf{Rep}}
\nc{\Rings}{\mathbf{Rings}} \nc{\Sets}{\mathbf{Sets}}
\nc{\DT}{\mathbf{DT}}

%%%%%%%%%%%%%%%%%%%Bbb fonts
\nc{\BA}{{\mathbb A}} \nc{\CC}{{\mathbb C}} \nc{\DD}{{\mathbb D}}
\nc{\EE}{{\mathbb E}} \nc{\FF}{{\mathbb F}} \nc{\GG}{{\mathbb G}}
\nc{\HH}{{\mathbb H}} \nc{\LL}{{\mathbb L}} \nc{\NN}{{\mathbb N}}
\nc{\QQ}{{\mathbb Q}} \nc{\RR}{{\mathbb R}} \nc{\BS}{{\mathbb{S}}} \nc{\TT}{{\mathbb T}}
\nc{\VV}{{\mathbb V}} \nc{\ZZ}{{\mathbb Z}}

%%%%%%%%%%%%%%%%%%% cal fonts

\nc{\calao}{{\mathcal A}} \nc{\cala}{{\mathcal A}}
\nc{\calc}{{\mathcal C}} \nc{\cald}{{\mathcal D}}
\nc{\cale}{{\mathcal E}} \nc{\calf}{{\mathcal F}}
\nc{\calfr}{{{\mathcal F}^{\,r}}} \nc{\calfo}{{\mathcal F}^0}
\nc{\calfro}{{\mathcal F}^{\,r,0}} \nc{\oF}{\overline{F}}
\nc{\calg}{{\mathcal G}} \nc{\calh}{{\mathcal H}}
\nc{\cali}{{\mathcal I}} \nc{\calj}{{\mathcal J}}
\nc{\call}{{\mathcal L}} \nc{\calm}{{\mathcal M}}
\nc{\caln}{{\mathcal N}} \nc{\calo}{{\mathcal O}}
\nc{\calp}{{\mathcal P}} \nc{\calq}{{\mathcal Q}} \nc{\calr}{{\mathcal R}}
\nc{\calt}{{\mathcal T}} \nc{\caltr}{{\mathcal T}^{\,r}}
\nc{\calu}{{\mathcal U}} \nc{\calv}{{\mathcal V}}
\nc{\calw}{{\mathcal W}} \nc{\calx}{{\mathcal X}}
\nc{\CA}{\mathcal{A}}

%%%%%%%%%%%%%%%%%%  frak fonts
\nc{\fraka}{{\mathfrak a}} \nc{\frakB}{{\mathfrak B}}
\nc{\frakb}{{\mathfrak b}} \nc{\frakd}{{\mathfrak d}}
\nc{\oD}{\overline{D}}
\nc{\frakF}{{\mathfrak F}} \nc{\frakg}{{\mathfrak g}}
\nc{\frakm}{{\mathfrak m}} \nc{\frakM}{{\mathfrak M}}
\nc{\frakMo}{{\mathfrak M}^0} \nc{\frakp}{{\mathfrak p}}
\nc{\frakS}{{\mathfrak S}} \nc{\frakSo}{{\mathfrak S}^0}
\nc{\fraks}{{\mathfrak s}} \nc{\os}{\overline{\fraks}}
\nc{\frakT}{{\mathfrak T}}
\nc{\oT}{\overline{T}}
%\nc{\frakx}{{\mathfrak x}}
\nc{\frakX}{{\mathfrak X}} \nc{\frakXo}{{\mathfrak X}^0}
\nc{\frakx}{{\mathbf x}}
%\nc{\frakTxo}{{\frakTx}^0}
\nc{\frakTx}{\frakT}      %All rooted trees, correspond to \ncsha(X)
\nc{\frakTa}{\frakT^a}        % rooted trees for \ncsha(A)
\nc{\frakTxo}{\frakTx^0}   % rooted trees for \ncshao(X)
\nc{\caltao}{\calt^{a,0}}   % rooted trees for \ncshao(A)
\nc{\ox}{\overline{\frakx}} \nc{\fraky}{{\mathfrak y}}
\nc{\frakz}{{\mathfrak z}} \nc{\oX}{\overline{X}}

\font\cyr=wncyr10

\nc{\al}{\alpha}
\nc{\lam}{\lambda}
\nc{\lr}{\longrightarrow}

%%%%%%%%%%%%%%%%%%%%%%%%%%%%%%%%%%%%%%%%%%%%%%%%%%%%%%%%%%%%%%%%%%

\title[Bialgebras for Rota-Baxter algebras]{
Bialgebras, Frobenius algebras and associative Yang-Baxter equations for Rota-Baxter algebras}

\author{Chengming Bai}
\address{Chern Institute of Mathematics \& LPMC, Nankai University, Tianjin 300071, China}
         \email{baicm@nankai.edu.cn}

\author{Li Guo}
\address{Department of Mathematics and Computer Science, Rutgers University, Newark, NJ 07102, USA}
         \email{liguo@rutgers.edu}

\author{Tianshui Ma}
\address{School of Mathematics and Information Science, Henan Normal University, Xinxiang 453007, China}
\email{matianshui@htu.edu.cn}

\date{\today}

\begin{abstract}
Rota-Baxter operators and bialgebras go hand in hand in their
applications, such as in the Connes-Kreimer approach to
renormalization and the operator approach to the classical
Yang-Baxter equation. We establish a bialgebra structure that is
compatible with the Rota-Baxter operator, called the Rota-Baxter
antisymmetric infinitesimal (ASI) bialgebra. This bialgebra is
characterized by generalizations of matched pairs of algebras and
double constructions of Frobenius algebras to the context of
Rota-Baxter algebras. The study of the coboundary case leads to an
enrichment of the associative Yang-Baxter equation (AYBE) to
Rota-Baxter algebras. Antisymmetric solutions of the equation are
used to construct Rota-Baxter ASI bialgebras. The notions of an
$\mathcal{O}$-operator on a Rota-Baxter algebra and a Rota-Baxter
dendriform algebra are also introduced to produce solutions of the
AYBE in Rota-Baxter algebras and thus to provide Rota-Baxter ASI
bialgebras. An unexpected byproduct is that a
Rota-Baxter \asi bialgebra of weight zero gives rise to a
quadri-bialgebra instead of bialgebra constructions for the dendriform algebra.
\end{abstract}

\subjclass[2010]{
16T10,   %bialgebra
16T25,   %Yang-Baxter equation
17B38,  %Yang-Baxter equations and Rota-Baxter operators
16W99,  %rings and algebras with additional structure/none of above, but in this section
17A30,  %algebras satisfying other identities
16T05,  %Hopf algebras and their applications
17B62,   %Lie bialgebras, Lie coalgebras
57R56,   %Topological quantum field theories
81R60   %noncommutative geometry
}

   %81T45,   %Topological field theories

\keywords{Rota-Baxter algebra; antisymmetric infinitesimal bialgebra; associative Yang-Baxter equation; $\mathcal{O}$-operator; dendriform algebra}

\maketitle

\vspace{-1.2cm}

\tableofcontents

\vspace{-1cm}

\allowdisplaybreaks

\section{Introduction}
This paper develops a bialgebra theory for Rota-Baxter algebras by applying methods from double constructions of Frobenius algebras, associative Yang-Baxter equations, $\calo$-operators and dendriform algebras.

\subsection{Hopf algebras and Lie bialgebras}

A bialgebra structure is the coupling of an algebra structure and
a coalgebra structure by certain compatibility conditions,
sometimes called distributive laws. The most well-known such
structures are the Hopf algebra~\mcite{Ab} for associative
algebras and the Lie bialgebra~\mcite{CP,D} for Lie algebras. The
central importance of bialgebras rests on their
connection with other structures arising from mathematics and
physics. Hopf algebra has its origin in topology, and serves as universal enveloping algebras of Lie algebras and quantization in terms of
quantum groups and Yang-Baxter equation.

Lie bialgebras also arise naturally in the study of Yang-Baxter equations,
and are closely related to quantum groups as deformations of
universal enveloping algebras. Etingof and
Kazhdan~\mcite{EK} proved that every Lie bialgebra has a
corresponding quantized universal enveloping algebra, that is,
there exists a quantization for every Lie bialgebra.

Further, bialgebras naturally lead to other important algebraic
structures of independent importance. Lie bialgebras are
characterized by Manin triples of Lie algebras, while Manin
triples of Lie algebras with respect to (symmetric) nondegenerate
invariant bilinear form  give constructions of self-dual (or
quadratic) Lie algebras which are useful in conformal field theory and quantized universal algebras~\mcite{Fu,Pe}. In addition, the Manin triple approach of the Lie
bialgebra is interpreted as cocycle conditions which provide a
rich structure theory and effective constructions for the Lie
bialgebra. In particular, the study of coboundary case leads to
the introduction of the classical Yang-Baxter equation (CYBE). The
(antisymmetric) solutions of the CYBE give rise to Lie bialgebras,
whereas $\mathcal O$-operators on Lie algebras introduced in
\cite{Ku}, which are a natural generalization of Rota-Baxter
operators, provide the needed solutions of the CYBE.

Similarly, Manin triples of Lie algebras with respect to
a (skew-symmetric) nondegenerate 2-cocycle, which are called
the parak\"ahler structures in geometry or phase
spaces in mathematical physics, correspond to left-symmetric
bialgebras~\mcite{Bai3,Ka,Ku2}.

The associative analog of the Lie bialgebra is the antisymmetric
infinitesimal (\asi) bialgebra, for which the Manin triples of
associative algebras with respect to (symmetric) nondegenerate
invariant bilinear form are called double constructions of
Frobenius algebras, the latter being widely applied to areas such
as 2d topological quantum field theory and string
theory~\mcite{Ko,LP}.
Further details can be
found in~\mcite{Ag2,Bai1} for example.

\vspace{-.2cm}
\subsection{Bialgebra structures of Rota-Baxter algebras}

The study of Rota-Baxter algebras originated from the work of the mathematician G.
Baxter~\mcite{Baxter} in probability and, in its early study,
attracted the attention of well-known analysts and combinatorists including
Atkinson, Cartier and Rota~\mcite{Ro1}. In the
Lie algebra context, the defining relation of Rota-Baxter operators is precisely the operator form of the classical Yang-Baxter equation, named after C.-N.~Yang and R.~J.~Baxter. This area has expanded
tremendously in recent years with its broad
applications and connections, notably in the Connes-Kreimer approach to renormalization of quantum field theory, where Hopf algebras and Rota-Baxter algebras are the fundamental structures, especially in the algebraic Birkhoff factorization. See~\mcite{Bai2,Br,CK,Guo1,GPZ} for further details.

Given the joint roles played by bialgebras and Rota-Baxter algebras in their applications, it is desirable to study their composed structures. Some progresses have been made in this direction.
There has been quite much study equipping free Rota-Baxter algebras with a bialgebra or Hopf algebra structures~\mcite{AGKN,EG,GGZ2}, thus providing rich structures on Rota-Baxter algebras and leading to connections with combinatorics and number theory~\mcite{Ho,YGT}.
There is also a notion of Rota-Baxter bialgebra~\mcite{ML}, as a quintuple $(A, \cdot, \Delta, P, Q)$ where $(A,\cdot,\Delta)$ is a bialgebra, $(A, \cdot, P)$ is a Rota-Baxter algebra of certain weight and $(A,\Delta, Q)$ is a Rota-Baxter coalgebra of another weight.
Yet, compatibility relations between the operators $P$ and $Q$ are still needed to have a bialgebra theory for Rota-Baxter algebras.

In this paper, we establish a bialgebra theory for Rota-Baxter
algebras by extending the approach of double construction of
Frobenius algebras to \asi bialgebras, after resolving challenges
imposed by the extra restrictions of the Rota-Baxter operators.
The method is quite general and might be adapted for the other
operators such as the differential, Nijenhuis, average and
Reynalds operators.
Note that these structures are different from Hom-Lie algebras~\mcite{HLS} which also have an extra linear operator, but the operator entails a twist of the existing Jacobian identity rather than adding a new relation to the structure.

To give a thorough and uniform
treatment of the possible compatibility conditions among the multiplication, comultiplication and the linear operators, we introduce an admissibility condition between a linear operator
and a Rota-Baxter algebra. Then we are able to apply the Rota-Baxter operator to the matched pairs of algebras, Frobenius algebras, associative Yang-Baxter equation, $\calo$-operators and dendriform algebras, so that the full theory of \asi bialgebras can be extended to Rota-Baxter algebras.
\vspace{-.2cm}
\subsection{Outline of the paper}
The main notions and constructions in this paper are
summarized in the following diagram.
$$ 
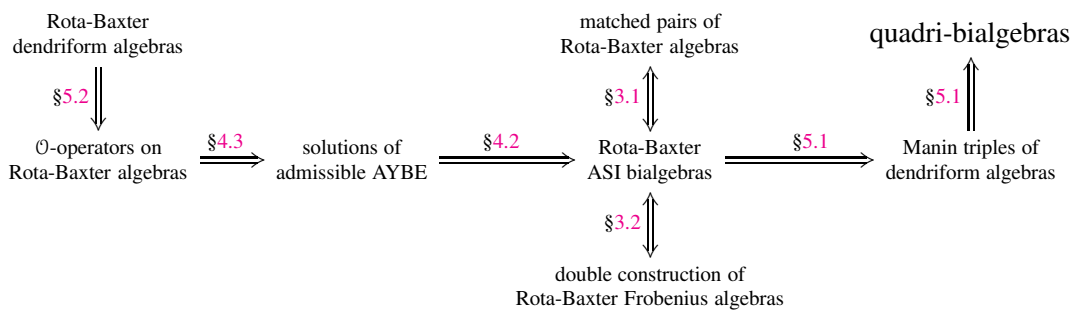
\xymatrix{
\text{Rota-Baxter}\atop
 \text{dendriform algebras} \ar@2{->}_{\S\ref{ss:rbdend}}[d] &&\text{matched pairs of}\atop \text{Rota-Baxter algebras} &\text{\small{quadri-bialgebras}}\\
\mathcal{O}\text{-operators on}\atop\text{Rota-Baxter algebras} \ar@2{->}^{\S\ref{ss:oop}}[r]&
\text{solutions of}\atop \text{admissible AYBE}
\ar@2{->}^{\S\ref{ss:ybe}}[r]& \text{Rota-Baxter}\atop \text{\asi bialgebras} \ar@2{->}^{\S\ref{ss:dend}}[r] \ar@2{<->}_{\S\ref{ss:double}}[d] \ar@2{<->}^{\S\ref{ss:match}}[u]&
\text{Manin triples of}\atop \text{dendriform algebras} \ar@2{->}^{\S\ref{ss:dend}}[u] & \\
&& \text{double construction of}\atop \text{Rota-Baxter
Frobenius algebras} & }
$$

In Section~\mref{sec:rbasi}, we introduce the notion of a Rota-Baxter antisymmetric infinitesimal (ASI)  bialgebra together with some preliminary examples. This notion is built on Rota-Baxter operators on an ASI bialgebra, but with a special set of compatibility conditions between the unary and binary operations that is the key to a complete bialgebra theory comparable to those of Lie bialgebras and ASI bialgebras.
To gain a good understanding of these compatibility conditions for
later applications, we conceptualize these conditions in
Section~\mref{ss:rep}, to the notion of an admissible quadruple of
a Rota-Baxter algebra, as the compatibility of a Rota-Baxter
algebra and a linear operator in terms of dual representations.

With the preparation in Section~\mref{ss:rep}, in Section~\mref{sec:match}, we give the general notion of a
matched pair of Rota-Baxter algebras, before specializing to
the case when the underlying linear spaces of the two Rota-Baxter algebras are dual to each other, providing an equivalent condition for Rota-Baxter \asi bialgebras (Theorem~\mref{thm:rbinfbialg1}). Then this notion is tied in with the notion of a double
construction of Rota-Baxter Frobenius algebra
(Theorem~\mref{thm:3.7}) and thus provides another equivalent
condition of Rota-Baxter \asi bialgebras
(Theorem~\mref{thm:rbbial}).

The remaining part of the paper studies and constructs Rota-Baxter \asi bialgebras through their relationship with Rota-Baxter operations on associative Yang-Baxter equations, $\mathcal {O}$-operators and dendriform algebras.

In Section~\mref{sec:aybe}, we focus on coboundary Rota-Baxter
ASI bialgebras. From their characterization equations
(Corollary~\ref{thm:4.6a}), we extract the notion of an admissible associative Yang-Baxter equation in a Rota-Baxter
algebra, whose study further leads to the notion of an
$\mathcal{O}$-operator on a Rota-Baxter algebra. Thus such
$\calo$-operators provide the needed solutions of the
admissible associative Yang-Baxter equation to give rise to
Rota-Baxter \asi bialgebras (Theorem~\mref{thm:4.21}). Several cases are considered where the conditions and computations can be made explicit, providing examples for the general results (Proposition~\ref{cor:cons}).
In particular, a Rota-Baxter operator on any Rota-Baxter algebra is naturally an $\mathcal O$-operator on this
Rota-Baxter algebra and hence produces a Rota-Baxter \asi bialgebra (Corollary~\mref{cor:4.23}).

In view of the role played by dendriform algebras~\mcite{Lo} in
the study of \asi bialgebras~\mcite{Bai1},
especially in constructing $\mathcal O$-operators, we introduce in
Section~\mref{sec:dend} the notion of a Rota-Baxter dendriform
algebra as a dendriform algebra with a Rota-Baxter operator, which gives a natural $\mathcal O$-operator on the associated Rota-Baxter algebra
(Proposition~\mref{pro:5.3}). Hence there is a construction of
Rota-Baxter \asi bialgebras from Rota-Baxter dendriform algebras
via such $\mathcal O$-operators, applying the methods introduced in the
previous section. Especially, any Rota-Baxter algebra of
weight zero already carries a natural Rota-Baxter dendriform algebra
structure and hence produces
a Rota-Baxter \asi bialgebra (Corollary~\mref{cor:construction}).
This phenomenon, of obtaining a nontrivial bialgebra structure
directly from the base algebraic structure as displayed here and
in the above Corollary~\mref{cor:4.23}, is new in Rota-Baxter
algebras, not found in previously considered algebra structures
such as Lie algebras and associative algebras. We further found that although a Rota-Baxter algebra of weight zero gives a
dendriform algebra~\cite{Ag}, a Rota-Baxter \asi bialgebra of weight zero does not give a bialgebra construction for the dendriform algebra, but rather a quadri-bialgebra~\mcite{NB}
(Remark~\mref{rmk:4.26a} and Corollary~\mref{co:quad}).

\smallskip

\noindent

{\bf Notations: }
Throughout this paper, we fix a field $K$. All vector spaces, tensor products, and
 linear homomorphisms are over $K$.  After Section~\mref{sec:rbasi}, all the vector spaces and algebras are finite dimensional unless otherwise specified. By an algebra, we mean an associative algebra not necessarily having a unit.
\vspace{-.2cm}
\section{Rota-Baxter ASI bialgebras and their admissibility conditions}
\mlabel{sec:rbasi} \vspace{-.2cm} In this section, we first
introduce the notion of a Rota-Baxter antisymmetric infinitesimal
(\asi) bialgebra as an enrichment of \asi bialgebras, with a key
role played by representations of a Rota-Baxter algebra on a dual
space. We then give a focused study of such representations for better understanding of Rota-Baxter \asi
bialgebras and their applications. \vspace{-.2cm}
\subsection{Rota-Baxter ASI bialgebras}
\mlabel{ss:rbasi} We first recall the notion of antisymmetric
infinitesimal bialgebras~\mcite{Ag2,Bai1,JR} as the associative analog of Lie bialgebras~\mcite{CP,D}.

\begin{defi}
\mlabel{de:bial}
An {\bf antisymmetric infinitesimal bialgebra} or
simply an {\bf \asi bialgebra} is a triple $(A,\cdot,\Delta)$
consisting of a vector space $A$ and linear maps $\cdot: A\ot A\to
A$ and $\Delta:A\to A\ot A$ such that
\begin{enumerate}
\item
the pair $(A,\cdot)$ is an associative algebra,
\item
the pair $(A,\Delta)$ is a coassociative coalgebra, and \item
with the flip map $\sigma:A\ot A\to A\ot A$, the following equations hold.
 \begin{equation}
 \Delta(a\cdot b)=(R_A(b)\otimes \id)\Delta(a)+(\id\otimes L_A(a))\Delta(b),
 \mlabel{eq:3.14}
 \end{equation}
 \begin{equation}
 (L_A(a)\otimes \id-\id\otimes R_A(a))\Delta(b)=\sigma(\id\otimes R_A(b)-L_A(b)\otimes \id)\Delta(a), \quad \forall a, b\in A.
 \mlabel{eq:3.15}
 \end{equation}
\end{enumerate}
\end{defi}
The terms infinitesimal and antisymmetric come from Eqs.~(\mref{eq:3.14}) and (\mref{eq:3.15}) respectively.
When $A$ is finite dimensional, an \asi bialgebra can be equivalently formulated in terms of a matched pair of algebras and a double construction of Frobenius algebras.

We will extend the theory of \asi bialgebras to the context of Rota-Baxter algebras.
\begin{defi}
\mcite{Baxter}
Let $\lambda\in K$ be given.  A {\bf Rota-Baxter algebra of weight $\lambda$} is a pair $(A,P)$, consisting of an algebra $A$ and a linear operator $P:A\to A$ such that
\begin{equation}
 P(a)P(b)=P(a P(b))+P(P(a)b)+\lambda P(ab), \quad \forall a, b \in A.
\mlabel{eq:1.1}
\end{equation}
Such a linear operator $P$ is called a {\bf Rota-Baxter operator of weight $\lambda$} on $A$. We will suppress the weight if there is no danger of confusion.
\end{defi}

Dualizing the notion of a Rota-Baxter algebra, we have

\begin{defi} \mcite{EF,JZ,ML}
\mlabel{de:corb}
A {\bf Rota-Baxter coalgebra} of weight $\lambda\in K$ is a triple $(A,\Delta,Q)$ where $(A,\Delta)$ is a coalgebra and $Q:A\to A$ is a linear operator such that
\begin{eqnarray}\mlabel{eq:corb}
 (Q\ot Q)\Delta(a)=(Q\ot \id)\Delta Q(a)+(\id\ot Q)\Delta Q(a)+\lambda \Delta Q(a), \quad \forall a\in A.
 \end{eqnarray}
\end{defi}

Any algebra or coalgebra naturally comes with Rota-Baxter
operators given by the scalar multiplications. Thus Rota-Baxter
algebras and coalgebras can be regarded as generalizations of
algebras and coalgebras. Generalizing the well-known duality
between a coalgebra and an algebra~\mcite{Ab}, for a vector space
$A$ and linear maps $Q:A\to A, \Delta:A\to  A\ot A$, if
$(A,\Delta,Q)$ is a Rota-Baxter coalgebra, then
$(A^*,\Delta^*,Q^*)$ is a Rota-Baxter algebra. When $A$ is finite
dimensional, the converse is also true~\mcite{JZ}.

We now introduce the main notion of the paper as an enrichment and generalization of \asi bialgebra in the context of Rota-Baxter algebras.

\begin{defi}
A {\bf Rota-Baxter antisymmetric infinitesimal (\asi) bialgebra} is a quintuple $(A,\cdot,\Delta,P,Q)$ or simply $((A,P),\Delta,Q)$ consisting of a vector space $A$ and linear maps
$$\cdot: A\ot A\to A,\ \ \Delta:A\to A\ot A,\ \  P,Q:A\to A$$
such that
\begin{enumerate}
\item
$(A,\cdot,\Delta)$ is an antisymmetric infinitesimal bialgebra,
\mlabel{it:rbasi1}
\item
$(A,\cdot, P)$ is a Rota-Baxter algebra,
\mlabel{it:rbasi2}
\item
$(A,\Delta,Q)$ is a Rota-Baxter coalgebra, and
\mlabel{it:rbasi3}
\item
the following compatibility conditions hold.
\begin{eqnarray}
 Q(aP(b))&=&Q(a)P(b)+Q(Q(a)b)+\lambda Q(a)b,
 \mlabel{eq:pduqr} \\
 Q(P(a)b)&=&P(a)Q(b)+Q(aQ(b))+\lambda aQ(b), \quad \forall a, b\in A,
 \mlabel{eq:pduql}\\
 (\id\ot Q)\Delta P&=&(P\ot Q)\Delta+(P\ot \id)\Delta P+\lambda(P\ot \id)\Delta,
\mlabel{eq:pduqrd}\\
(Q\ot \id)\Delta P&=&(Q\ot P)\Delta+(\id\ot P)\Delta P+\lambda(\id\ot P)\Delta.
\mlabel{eq:pduqld}
 \end{eqnarray}
\end{enumerate}
\mlabel{de:rbbial}
\end{defi}

We give some preliminary examples. More substantial examples will be provided later after the needed tools are developed.

\begin{ex}
As noted above, for a given scalar $\lambda$ in the base field, the scalar product by $-\lambda$ is a Rota-Baxter operator of weight $\lambda$. Thus for any algebra $R$, $(R,-\lambda\id)$ is a Rota-Baxter algebra of weight $\lambda$. Similarly, for any coalgebra $C$, the pair $(C,-\lambda\id)$ is a Rota-Baxter coalgebra of weight $\lambda$. It then follows that for
 any \asi bialgebra $(R,\cdot,\Delta)$, together with the scalar product operators $P=Q=-\lambda \id$, the quintuple $(A,\cdot,\Delta,P,Q)$ is a Rota-Baxter \asi bialgebra of weight $\lambda$.
\mlabel{ex:scalar}
\end{ex}

\begin{ex}\label{ex:directsum}
A basic example of Rota-Baxter algebras of weight $-1$ is the
direct sum $A\oplus B$ of two algebras with the projection
operator to either of the two summands. The same construction
gives Rota-Baxter coalgebras of weight $-1$. Now let
$(A,\cdot_A,\Delta_A)$ and $(B,\cdot_B,\Delta_B)$ be \asi
bialgebras. It is direct to check that the direct sum $A\oplus B$
with the component operations $\cdot_A+ \cdot_B$ and $\Delta_A+
\Delta_B$ is again an \asi bialgebra. Let $P_A, P_B:A\oplus B\to
A\oplus B$ be the projections to $A$ and $B$ respectively. Then the quintuple $(A\oplus B, \cdot_A+ \cdot_B,\Delta_A+
\Delta_B,P_A,P_B)$ is a Rota-Baxter \asi bialgebra. For this it
remains to check the compatibility conditions in
Eqs.~\eqref{eq:pduqr}-\eqref{eq:pduqld}. For 
$x\in A, a\in B$, we have
\begin{eqnarray*}
&&(\id_{A\oplus B}\ot P_B)(\Delta_A+ \Delta_B) P_A(x+a)=
(P_A\ot P_B)(\Delta_A+ \Delta_B)(x+a)=0,\\
&&(P_A\ot \id_{A\oplus B})(\Delta_A+\Delta_B) P_A(x+a)=
(P_A\ot \id_{A\oplus B})(\Delta_A+
\Delta_B)(x+a)=\Delta_A(x).
\end{eqnarray*}
Hence
\begin{eqnarray*}
(\id_{A\oplus B}\ot P_B)(\Delta_A+ \Delta_B) P_A&=& (P_A\ot
P_B)(\Delta_A+ \Delta_B)+ (P_A\ot \id_{A\oplus
B})(\Delta_A+\Delta_B) P_A\\&\mbox{}&- (P_A\ot \id_{A\oplus
B})(\Delta_A+ \Delta_B), \end{eqnarray*}
that is,
Eq.~\meqref{eq:pduqrd} holds. The other equations can be verified in the same way.
\end{ex}

\begin{rmk}
As we can see in Definition~\mref{de:rbbial}, the key to the notion of a Rota-Baxter \asi algebra
is a suitable choice of compatibility conditions among the
multiplication, comultiplication and the linear operators. Our
choice of the conditions in
Eqs.~\meqref{eq:pduqr}-\meqref{eq:pduqld} was motivated by our goal of developing a full theory expanding those of Lie bialgebras (resp. \asi bialgebras), including their close connections with matched pairs, Manin triples (resp. double constructions of Frobenius algebras) and furthermore with classical (resp. associative) Yang-Baxter equations, $\calo$-operators and pre-Lie (resp. dendriform) algebras. These connections and
applications of Rota-Baxter \asi algebras will be presented in the later sections of the paper. For this purpose, we first need a conceptual understanding of these compatibility conditions which we will explore in the next subsection.
\mlabel{rk:comp}
\end{rmk}

\subsection{Rephrasing Rota-Baxter \asi bialgebras in terms of
    admissible quadruples}
\mlabel{ss:rep}
We will work with the representation theory of an algebra in the sense of bimodules, instead of one-sided modules.

\begin{defi}
Let $A$ be an algebra. A {\bf representation of $A$} or an {\bf $A$-bimodule} is a triple $(V,\ell,r)$, abbreviated as $V$, consisting of a vector space $V$ and linear maps
$$\ell, r: A\to  \End_K(V)$$
such that
$$
\ell(a)(\ell(b)v)=\ell(ab)v,\
(vr(a))r(b)=vr(ab),\
 (\ell(a)v)r(b)=\ell(a)(vr(b)), \  \forall a, b\in A, v\in V.
$$
\end{defi}
With $L=L_A$ and $R=R_A$ denoting the left and right multiplications by elements of $A$ respectively, the triple $(A, L, R)$ is a representation of $A$, called the {\bf adjoint representation} of $A$.

Let $\ell, r: A\to \End_K(V)$ be linear maps.
Define a multiplication on $A\oplus V$ by
 \begin{equation}
 (a+u)\cdot (b+v):=ab+(\ell(a)v+u r(b)),\quad \forall a, b\in A, u, v \in V.
 \mlabel{eq:2.3}
 \end{equation}
Then as is well known, $A\oplus V$ is an
algebra, denoted by $A\ltimes_{\ell,r} V$ and called
the {\bf semi-direct product} of $A$ by $V$, if and only if $(V,\ell,r)$ is a
representation of $A$.

Extending the concept of a representation of an algebra to the context of Rota-Baxter algebras, we
give

\begin{defi}
A {\bf representation}  or a {\bf bimodule} of a Rota-Baxter
algebra $(A, P)$ of weight $\lambda$ is a quadruple $(V, \ell, r,
\alpha)$ where $(V, \ell, r)$ is an $A$-bimodule and $\alpha$ is a
linear operator on $V$ such that
 \begin{equation}
 \ell(P(a))\alpha(v)=\alpha(\ell(P(a))v)+\alpha(\ell(a)\alpha(v))+\lambda \alpha(\ell(a)v))
\mlabel{eq:2.1}
 \end{equation}
 and
 \begin{equation}
 \alpha(v)r(P(a))=\alpha(\alpha(v)r(a))+\alpha(v r(P(a)))+\lambda \alpha(v r(a)), \quad \forall a\in A, v\in V.
\mlabel{eq:2.2}
 \end{equation}
Two representations $(V_1, \ell_1, r_1, \alpha_1)$ and $(V_2,
\ell_2, r_2, \alpha_2)$ of a Rota-Baxter algebra $(A, P)$ are
called {\bf equivalent} if there exists a linear
isomorphism $\varphi:V_1\rightarrow V_2$ such that
 \begin{equation} \mlabel{de:2.1}
 \varphi (\ell_1(a)(v))=\ell_2(a)\varphi(v), \varphi(vr_1(a))=\varphi(v)r_2(a), \varphi \alpha_1 (v)=\alpha_2\varphi(v) ,\;\;\forall a\in A, v\in V_1.
 \end{equation}
 \end{defi}
It follows immediately from the definition that the quadruple $(A,L,R,P)$ is a representation of the Rota-Baxter algebra $(A,P)$, called the {\bf adjoint representation} of $(A,P)$.

For representations of Rota-Baxter algebras in the sense of one-sided modules, we refer the reader to~\mcite{QGG,QP}.
Semi-direct products  of Rota-Baxter algebras can also be characterized by their representations.

\begin{pro}
Let $(A,P)$ be a Rota-Baxter algebra of weight $\lambda$.  Let
$(V, \ell, r)$ be a representation of the algebra $A$ and let
$\alpha$ be a linear operator on $V$. Define a linear map
 \begin{equation}
 P_{A\oplus V}: A\oplus V\to  A\oplus V, \quad  P_{A\oplus V}(a+u):=P(a)+\alpha(u),
\mlabel{eq:2.4}
 \end{equation}
which is simply denoted by $P_{A\oplus V}:=P+\alpha$.
 Then together with the multiplication defined in Eq.~$($\mref{eq:2.3}$)$, $(A\oplus V, P_{A\oplus V})$ is a Rota-Baxter algebra of weight $\lambda$ if and only if $(V,\ell,r,\alpha)$ is a representation of $(A,P)$.
 The resulting Rota-Baxter algebra is denoted by $(A\ltimes_{\ell,r} V, P+\alpha)$ and is called the {\bf semi-direct product} of $(A,P)$ by its representation $(V,\ell,r,\alpha)$.
\mlabel{pro:2.2}
 \end{pro}

The proof will be omitted since this result is a special case of
the matched pairs of Rota-Baxter algebras in
Theorem~\mref{thm:3.1}, when $B=V$ is equipped with the zero
multiplication.

We fix more notations. Denote the usual pairing between the dual space $V^*$ and $V$ by
 $$
 \langle\, ,\, \rangle : V^*\times V\to  K, ~~\langle v^*, v \rangle :=v^*(v), \quad\forall v\in V, v^*\in V^*.
 $$
 For a linear map $\varphi: V\to  W$, we denote the transpose map by $\varphi^*: W^*\to  V^*$
 given by
 $$
 \langle \varphi^*(w^*), v \rangle =\langle w^*, \varphi(v) \rangle , \quad \forall v\in V, w^*\in W^*.
 $$

For a representation $(V,\ell, r)$ of an algebra $A$, define the triple $(V^*, r^*, \ell^*)$ where the linear maps $\ell^*, r^*: A\to  \End_K(V^*)$ are defined by
 \begin{equation}
 \langle v^* \ell^*(a), v \rangle =\langle v^*, \ell(a) v \rangle ,~~\langle r^*(a) v^*, v \rangle =\langle v^*, v r(a) \rangle, \ \
   \forall a\in A, v^*\in V^*, v\in V.
\mlabel{eq:2.5} \end{equation} Then the triple $(V^*, r^*,
\ell^*)$ is again a representation of $A$~\cite[Lemma
2.1.2]{Bai1}, called the {\bf dual representation} of $(V, \ell,
r)$. However, this property does not hold for representations of
Rota-Baxter algebras, that is, the linear dual of a
representation of a Rota-Baxter algebra is not necessarily a Rota-Baxter representation. In fact, the following extra conditions are needed.

\begin{lem}
\mlabel{it:2.3a}
\mlabel{lem:admrep}
Let $(A, P)$ be a Rota-Baxter algebra of weight $\lambda$.  Let $(V, \ell, r)$ be a representation of the algebra $A$ and let $\beta:V\to V$ be a linear map.
The quadruple $(V^*,r^*,\ell^*,\beta^*)$ is a representation of $(A,P)$ if and only if the linear operator $\beta$ satisfies
\begin{eqnarray}
& \beta(vr(P(a)))-\beta(v)r(P(a))-\beta(\beta(v)r(a))-\lambda\beta(v)r(a)=0,& \mlabel{eq:it:2.3a} \\
& \beta(\ell(P(a))v)-\ell(P(a))\beta(v)-\beta(\ell(a)\beta(v))-\lambda \ell(a)\beta(v)=0,& \forall a\in A, v\in V.
\mlabel{eq:it:2.3b}
\end{eqnarray}
\end{lem}
\vspace{-.5cm}
\begin{proof}
Since $(V^*, r^*, \ell^*)$ is an $A$-bimodule, we just need to determine when $\beta^*$ satisfies Eqs.~(\mref{eq:2.1}) and (\mref{eq:2.2}) where $\alpha$ is replaced by $\beta^*$. 
By Eq.~(\mref{eq:2.5}), we have
\vspace{-.2cm}
\begin{eqnarray*}
&&\langle \ell(P(a))\beta^*(v^*)-\beta^*(\ell(P(a))v^*)-\beta^*(\ell(a)\beta^*(v^*))-\lambda \beta^*(\ell(a)v^*)), v \rangle \\
&=& \langle  v^*, \beta(vr(P(a)))-\beta(v)r(P(a))-\beta(\beta(v)r(a))-\lambda\beta(v)r(a)\rangle,
\ \forall a\in A, v\in V, v^*\in V^*.
\end{eqnarray*}
Thus Eq.~(\mref{eq:2.1}) is equivalent to
Eq.~(\mref{eq:it:2.3a}). The same argument proves
Eq.~(\mref{eq:it:2.3b}).
\end{proof}
\vspace{-.2cm}
We introduce a notion to conceptualize this key property.
\begin{defi}
    \mlabel{de:admop}
    Use the same data $(A,P), (V,\ell,r)$ and $\beta$ as in Lemma~\mref{lem:admrep}. If any (and hence both) of the equivalent conditions is satisfied, we say that $\beta$ is {\bf \admt the Rota-Baxter algebra $(A,P)$ on $(V,\ell,r)$}. To
    allow more flexibility in applying this notion, we also say that $(A,P)$ is
    {\bf $\beta$-admissible} on $(V,\ell,r)$ or that the quadruple
    $(V,\ell,r,\beta)$ is {\bf admissible}.
    When $(V,\ell,r)$ is taken to be the adjoint representation $(A,L,R)$ of the algebra $A$, we say that {\bf $\beta$ is \admt $(A,P)$} or simply {\bf $(A,P)$ is $\beta$-admissible}.
\end{defi}
\vspace{-.3cm}
Then we have
\begin{cor}
    Let $(A, P)$ be a Rota-Baxter algebra of weight $\lambda$. A linear operator
    $Q$ on $A$ is \admt $(A,P)$ if and
    only Eqs.~\eqref{eq:pduqr}-\meqref{eq:pduql} hold.
    \mlabel{cor:pqadmin}
\end{cor}
\vspace{-.2cm}
Further the definition of a Rota-Baxter \asi bialgebra can be rephrased as
\vspace{-.2cm}
\begin{pro}
    \mlabel{pro:asi2} A quintuple $(A,\cdot,\Delta,P,Q)$ with notations in Definition~\mref{de:rbbial} is a Rota-Baxter \asi bialgebra if and only if it satisfies conditions \eqref{it:rbasi1}-\eqref{it:rbasi3} in Definition~\mref{de:rbbial} and that
    $Q$ and $P^*$ are admissible to the Rota-Baxter algebras $(A,P)$ and $(A^*,Q^*)$ respectively.
\end{pro}
\vspace{-.3cm}
\begin{proof}
Of the compatibility conditions in
Eqs.~\meqref{eq:pduqr}-\meqref{eq:pduqld} for a Rota-Baxter \asi
bialgebras, Corollary~\mref{cor:pqadmin} already equals Eqs.~\meqref{eq:pduqr} and
\meqref{eq:pduql} to the admissibility of $Q$ to the
Rota-Baxter algebra $(A,P)$.
Further, Eqs.~\meqref{eq:pduqrd} and \meqref{eq:pduqld} can be
rewritten as
\vspace{-.2cm}
\begin{eqnarray}
P^*(uQ^*(v))&=&P^*(u)Q^*(v)+P^*(P^*(u)v)+\lambda P^*(u)v,
\mlabel{eq:pduqrddual}\\
P^*(Q^*(u)v)&=&Q^*(u)P^*(v)+P^*(uP^*(v))+\lambda uP^*(v), \quad \forall u, v\in A^*.
\mlabel{eq:pduqlddual}
\end{eqnarray}
This means that $P^*$ is admissible to the Rota-Baxter algebra $(A^*,Q^*)$.
\end{proof}
\section{Matched pairs of Rota-Baxter algebras, double constructions of Rota-Baxter Frobenius  algebras and Rota-Baxter ASI bialgebras}
\mlabel{sec:match}
%\vspace{-.2cm}
In this section, we introduce the notions of a matched pair of
Rota-Baxter algebras and a double construction of
Rota-Baxter Frobenius algebra.
Generalizing the characterizations of Lie bialgebras and
ASI bialgebras in terms of Manin triples
for Lie algebras or double constructions of Frobenius algebras for
associative algebras~\mcite{Ag2,Bai1,CP}, we prove that these new notions give equivalent conditions for a Rota-Baxter \asi
bialgebra.
\subsection{Matched pairs of Rota-Baxter algebras}
\mlabel{ss:match} We first recall the concept of a matched pair of algebras~\mcite{Bai1}.
\vspace{-.1cm}
\begin{defi}
A {\bf matched pair of algebras} consists of
algebras $(A, \opa)$ and $(B, \opb)$, together with
linear maps $\ell_A, r_A: A\to  \End_K(B)$ and $\ell_B,
r_B: B\to  \End_K(A)$ such that
\begin{enumerate}
\item
$(A, \ell_B, r_B)$ is a representation of $(B, \opb)$, \item
$(B, \ell_A, r_A)$ is a representation of $(A, \opa)$ and
\item
 the following compatibility conditions hold: for $a, a'\in A$ and $b, b'\in B$,
 \begin{eqnarray}
& \ell_A(a)(b\opb b')=\ell_A(a r_B(b))b'+(\ell_A(a)b)\opb b';&
 \mlabel{eq:3.1}\\
& (b\opb b')r_A(a)=b~r_A(\ell_B(b')a)+b\opb (b' r_A(a));&
 \mlabel{eq:3.2}\\
& \ell_B(b)(a\opa a')=\ell_B(b~r_A(a))a'+(\ell_B(b)a)\opa a';&
 \mlabel{eq:3.3}\\
& (a\opa a')r_B(b)=a~r_B(\ell_A(a')b)+a\opa (a'r_B(b));&
 \mlabel{eq:3.4}\\
& \ell_A(\ell_B(b)a)b'+(b~r_A(a))\opb b'=b~r_A(a~r_B(b'))+b\opb (\ell_A(a)b');&
 \mlabel{eq:3.5}\\
& \ell_B(\ell_A(a)b)a'+(a~r_B(b))\opa a'=a~r_B(b~r_A(a'))+a\opa (\ell_B(b)a').&
 \mlabel{eq:3.6}
\end{eqnarray}
\end{enumerate}
\mlabel{de:match}
\end{defi}
\vspace{-.2cm}
There is a characterization of \asi bialgebras by matched pairs of algebras.
\vspace{-.1cm}
\begin{thm} \cite{Bai1}\mlabel{thm:md}
Let $(A, \cdot)$ be an algebra. Suppose that there is an algebra
$(A^*, \circ)$ on the linear dual $A^*$. Let $\Delta:A\to
A\ot A$ be the linear dual of $\circ:A^*\ot A^*\to A^*$. Then $(A,\cdot, \Delta)$ is an \asi bialgebra if and only if
$(A,A^*, {R_\cdot}^*, {L_\cdot}^*, {R_\circ}^*,
{L_\circ}^*)$ is a matched pair of algebras.
\end{thm}

By~\mcite{Bai1}, for algebras $(A,\opa)$, $(B,\opb)$ and linear
maps $\ell_A, r_A: A\to \End_K(B)$, $\ell_B, r_B:B\to
\End_K(A)$, define a multiplication on the direct
sum $A\oplus B$ by
 \begin{equation}
 (a+b)\star (a'+b'):=(a\opa a'+a r_B(b')+\ell_B(b) a')+(b\opb b'+\ell_A(a) b'+b r_A(a')),
\mlabel{eq:3.7}
 \end{equation}
for $a, a' \in A$ and $ b, b'\in B$. Then $(A\oplus B, \star)$ is an algebra if and only if $((A,\opa),(B,\opb), \ell_A, r_A,$ $\ell_B,r_B)$ is a matched pair of $(A,\opa)$ and $(B,\opb)$.
We denote the resulting algebra $(A\oplus B,\star)$ by $A\bowtie_{\ell_A,r_A}^{\ell_B,r_B} B$ or simply
$A\bowtie B$.
Further, for any algebra $C$ whose underlying vector space is a
linear direct sum of two subalgebras $A$ and $B$, there is a
matched pair $(A,B,\ell_A,r_A,\ell_B,r_B)$ such that there is an
isomorphism from the resulting algebra $(A\oplus
B,\star)$ via Eq.~(\mref{eq:3.7}) to the algebra $C$ and the
restrictions of the isomorphism to $A$ and $B$ are the identity
maps.

We extend this property to Rota-Baxter algebras.

\begin{defi}
A {\bf matched pair of Rota-Baxter algebras} is a sextuple  $((A, P_A), (B, P_B),$ $\ell_A, r_A,$ $\ell_B, r_B)$ where $(A, P_A)$, $(B, P_B)$ are Rota-Baxter algebras, $(B, \ell_A, r_A, P_B)$ is a representation of $(A,P_A)$, $(A, \ell_B, r_B$, $P_A)$ is a representation of $(B, P_B)$,
and $(A,B, \ell_A,r_A,\ell_B,r_B)$ is a matched pair of algebras.
\end{defi}
\vspace{-.3cm}
\begin{thm}
Let $(A,P_A)$ and $(B,P_B)$ be Rota-Baxter algebras of weight $\lambda$ and let $(A,B, \ell_A,$ $ r_A,\ell_B, r_B)$ be a matched pair of the algebras $A$ and $B$. On the resulting algebra $A\bowtie B$ from Eq.~(\mref{eq:3.7}), define the linear map
 \begin{equation}
P_{A\oplus B}:A\bowtie B\to A\bowtie B,  \quad P_{A\oplus B}(a+b)=P_A(a)+P_B(b),\quad \forall a\in A, b\in B.
 \mlabel{eq:3.8}
 \end{equation}
Then the pair $(A\bowtie B,P_{A\oplus B})$ is a
Rota-Baxter algebra of weight $\lambda$ if and only if $((A,
P_A),$ $ (B, P_B), \ell_A, r_A,$ $\ell_B, r_B)$ is a matched pair
of the Rota-Baxter
algebras $(A,P_A)$ and $(B,P_B)$.
\mlabel{thm:3.1}
 \end{thm}
\vspace{-.2cm}
 \begin{proof}
Computing two sides of the Rota-Baxter equation~(\mref{eq:1.1})
for the operator $P_{A\oplus B}$, for $a, a'\in
A$ and $b, b'\in B$, on the left hand side of the equation we have
\begin{eqnarray*}
\lefteqn{P_{A\oplus B}(a+b)\star P_{A\oplus B}(a'+b')\stackrel{(\mref{eq:3.8})}{=}
(P_A(a)+P_B(b))\star (P_A(a')+P_B(b'))} \\
&\stackrel{(\mref{eq:3.7})}{=}&
P_A(a)\opa P_A(a')+P_A(a)r_B(P_B(b'))+\ell_B(P_B(b))P_A(a')\\
 &&+P_B(b)\opb P_B(b')+\ell_A(P_A(a))P_B(b')+P_B(b)r_A(P_A(a')),
\end{eqnarray*}
and on the right hand side of the equation, by a similar computation we obtain
{\small{
 \begin{eqnarray*}
&&P_{A\oplus B}\big(P_{A\oplus B}(a+b)\star (a'+b') +
(a+b)\star P_{A\oplus B}(a'+b') +\lambda (a+b)\star (a'+b')\big)\\
 &=&
 P_A\big(P_A(a)\opa a'+a\opa P_A(a')+\lambda a\opa a'\big) +P_A\big(P_A(a)r_B(b')+a r_B(P_B(b'))+\lambda a r_B(b')\big)\\
 && +P_A\big(\ell_B(P_B(b)) a'+\ell_B(b) P_A(a')+\lambda \ell_B(b) a'\big)
 +P_B\big(P_B(b)\opb b' +b\opb P_B(b')+\lambda b\opb b'\big)\\
&& +P_B\big(\ell_A(P_A(a)) b'+\ell_A(a)P_B(b')+\lambda \ell_A(a) b'\big)+P_B\big(P_B(b)r_A(a')+b r_A(P_A(a'))+\lambda b r_A(a')\big).
\end{eqnarray*}
}}
Now if $((A, P_A), (B, P_B), \ell_A, r_A,$ $\ell_B, r_B)$ is a matched pair of the Rota-Baxter algebras $(A,P_A)$ and $(B,P_B)$, then each term in the six-term sum of the left hand side equals the corresponding term of the right hand side. Therefore $P_{A\oplus B}$ is a Rota-Baxter operator of weight $\lambda$.

Conversely, suppose that $P_{A\oplus B}$ satisfies the Rota-Baxter
equation~(\mref{eq:1.1}). Taking $a'=b=0$ in the equation and
comparing, we obtain two of the four equalities in order for $((A,
P_A), (B, P_B),$ $ \ell_A, r_A,$ $\ell_B, r_B)$ to be a matched pair
of Rota-Baxter algebras. Likewise, taking $a=b'=0$, we obtain
another two of the four equalities in order to have a matched pair
of Rota-Baxter algebras.
\end{proof}

\begin{thm}
\mlabel{thm:rbinfbialg1}
Let $(A, \cdot,P)$ be a Rota-Baxter algebra. Suppose that there is
a Rota-Baxter algebra $(A^*, \circ,Q^*)$ on the linear dual
$A^*$ of $A$. Let $\Delta:A\rightarrow A\otimes A$ denote the linear dual of
the multiplication $\circ:A^*\ot A^*\to A^*$ on $A^*$, that is, $(A,\Delta,Q)$ is a Rota-Baxter coalgebra.
Then the quintuple
$(A,\cdot,\Delta,P,Q)$ is a Rota-Baxter \asi bialgebra if and only if
the sextuple $((A,P),(A^*,Q^*), {R_\cdot}^*, {L_\cdot}^*, {R_\circ}^*, {L_\circ}^*)$ is a matched pair of Rota-Baxter algebras.
\end{thm}

\begin{proof}
($\Longrightarrow$) If $(A,\cdot,\Delta,P,Q)$ is a Rota-Baxter \asi bialgebra, then
$(A,\cdot,\Delta)$ is an ASI
bialgebra and the linear operators $Q$ and $P^*$ are \admt $(A,\cdot, P)$ and
$(A^*,\circ, Q^*)$ respectively. The former means that $(A,A^*,{R_\cdot}^*, {L_\cdot}^*,{R_\circ}^*, {L_\circ}^*)$ is a
matched pair of algebras by Theorem~\mref{thm:md} and the latter means that $(A^*,R_\cdot ^*,L_\cdot^*,Q^*)$ is a representation of $(A,\cdot, P)$ and $(A,R_{\circ}^*,L_{\circ}^*,P)$ is a representation of $(A^*,\circ, Q^*)$. Hence $((A,P),(A^*,Q^*), {R_\cdot}^*, {L_\cdot}^*, {R_\circ}^*, {L_\circ}^*)$ is a matched pair of Rota-Baxter algebras.

($\Longleftarrow$)
By definition, if $((A,P),(A^*,Q^*), {R_\cdot}^*, {L_\cdot}^*, {R_\circ}^*, {L_\circ}^*)$ is a matched pair of Rota-Baxter algebras,
then $(A,A^*,{R_\cdot}^*, {L_\cdot}^*, {R_\circ}^*, {L_\circ}^*)$ is a
matched pair of algebras and the linear operators $Q$ and $P^*$ are \admt $(A,P)$ and
$(A^*,Q^*)$ respectively. Hence $(A,\cdot,\Delta,P,Q)$ is a Rota-Baxter \asi bialgebra.
\end{proof}
\vspace{-.3cm}
\subsection{Double constructions of Rota-Baxter Frobenius algebras}
\mlabel{ss:double}
We recall the concept of a double construction of Frobenius algebra. See~\mcite{Bai1} for details.
\vspace{-.1cm}
\begin{defi}
 A bilinear form $\frakB(\; ,\; )$ on an algebra $A$ is called {\bf invariant} if
\begin{equation}
 \mathfrak{B}(ab, c)=\mathfrak{B}(a, bc),~~\forall~a, b, c\in A.
 \mlabel{eq:1.3}
 \end{equation}
A {\bf Frobenius algebra} $(A,\frakB)$ is an algebra $A$ with a
nondegenerate invariant bilinear form $\frakB(\; ,\; )$. A Frobenius
algebra $(A,\frakB)$ is called {\bf symmetric} if $\frakB(\; ,\;
)$ is symmetric.
 \mlabel{de:1.4}
 \end{defi}
Let $(A,\cdot)$ be an algebra. Suppose that there is an algebra
structure $\circ$ on its dual space $A^\ast$, and an algebra
structure on the direct sum $A\oplus A^\ast$ of the underlying
vector spaces of $A$ and $A^\ast$ which contains both $(A,\cdot)$
and $(A^\ast,\circ)$ as subalgebras. Define a
bilinear form on $A\oplus A^*$ by
\vspace{-.2cm}
\begin{equation}
\frakB_d(x + a^* ,y + b^* ) = \langle x,b^*\rangle + \langle a^* , y\rangle, \quad \forall a^*, b^* \in A^* , x, y \in A.
\mlabel{eq:3.9}
\end{equation}
If $\frakB_d$ is invariant, so that
$(A\oplus A^*,\frakB_d)$ is a symmetric Frobenius algebra, then the Frobenius algebra is called a {\bf double construction of Frobenius algebra} associated to $(A,\cdot)$ and
$(A^\ast,\circ)$, which is denoted by $(A\bowtie A^* , \frakB_d)$.
The notation $A\bowtie A^*$ is justified since the algebra on
$A\oplus A^*$ comes from a matched pair from $A$ and $A^*$ in Eq.~(\mref{eq:3.7}). Indeed, we have

\begin{thm}
\cite[Theorem~2.2.1]{Bai1} Let $(A,\cdot)$ and $(A^*,\circ)$ be algebras.
Then there is a double construction of
Frobenius algebra associated to $(A,\cdot)$ and $(A^*,\circ)$ if
and only if $(A,A^*, R^*_\cdot, L^*_\cdot,R^*_\circ,L^*_\circ)$ is
a matched pair of algebras.
\mlabel{thm:frob}
\end{thm}

We now extend these notions and properties to Rota-Baxter algebras.

\begin{defi}\mlabel{de:1.3}
A {\bf Rota-Baxter Frobenius algebra} is a triple $(A,P,\frakB)$
where $(A,P)$ is a Rota-Baxter algebra and $(A,\frakB)$ is a Frobenius algebra.
Let $\hat{P}:A\to A$ denote the adjoint linear transformation of $P$ under the nondegenerate bilinear form $\frakB$:
\begin{equation}\mlabel{eq:adjoint}
\mathfrak{B}(P(a), b)=\mathfrak{B}(a, \hat P(b)),\;\;\forall a,b\in A.
\end{equation}
\end{defi}

It is remarkable that the symmetric Frobenius property of a Rota-Baxter
algebra $(A,P)$ naturally guarantees a representation on the dual
space $A^*$.

\begin{pro}
Let $(A,P,\frakB)$ be a Rota-Baxter symmetric Frobenius algebra. Then for the adjoint operator $\hat P$ in Eq.~\meqref{eq:adjoint},
the quadruple $(A^*, R^*,L^*,{\hat P}^*)$ is a representation of the Rota-Baxter algebra $(A,P)$ that is equivalent to $(A, L, R, P)$.

Conversely, let $(A,P)$ be a Rota-Baxter algebra and $Q:A\to
A$ be a linear map that is \admt $(A,P)$. If the resulting representation $(A^*, R^*,L^*,Q^*)$ of $(A,P)$ is equivalent to $(A,L,R,P)$, then there exists a nondegenerate bilinear form
$\frakB(\;,\;)$ such that $(A,P,\frakB)$ is a Rota-Baxter Frobenius algebra for which $\hat P=Q$.
\mlabel{pp:frobadm}
\end{pro}

\begin{proof}For $a,b,c\in A$, by the Rota-Baxter relation in Eq.~(\mref{eq:1.1}), we obtain
\begin{eqnarray*}
0&=& \frakB(P(a)P(b),c)-\frakB(P(aP(b)),c)-\frakB(P(P(a)b),c)-\frakB(\lambda P(ab),c)\\
&=& \frakB(P(a),P(b)c)- \frakB(aP(b),\hat P(c)) - \frakB(P(a)b, \hat P(c)) - \frakB(\lambda ab,\hat P(c))\\
&=& \frakB\Big(a,\hat P(P(b)c) - P(b)\hat P(c) - \hat P(b\hat P(c)) - \lambda b\hat P(c)\Big),
\end{eqnarray*}
yielding
$ \hat P(P(b)c)-P(b)\hat P(c)-\hat P(b\hat P(c))-\lambda b\hat P(c)=0.$
This gives Eq.~(\mref{eq:pduql}). Applying the symmetry of
$\frakB$, a similar argument gives Eq.~(\mref{eq:pduqr}). Hence
$(A^*, R^*,L^*,{\hat P}^*)$ is a representation of $(A,P)$. Define
a linear map $\phi:A\to  A^*$ by
$$\phi(a)(b):=\langle \phi(a), b\rangle=\frakB(a,b),\;\;\forall a,b\in A.$$
The nondegeneracy of $\frakB$ gives the bijectivity of $\phi$.
Also for $a,b,c\in A$, we have
\begin{eqnarray*}
&\phi (L(a)b)c=\frakB(ab,c)=\frakB(c,ab)=\frakB(ca,b)=\langle \phi(b), ca\rangle=\langle R^*(a)\phi(b), c\rangle=R^*(a)\phi(b) c
\end{eqnarray*}
and similarly $\phi(aR(b))c=\phi(a)L^*(b)c$ and $\phi(P(a))b=\hat{P}^*(\phi(a))b.$
Hence $(A, L, R, P)$ is equivalent
to $(A^*, R^*,L^*,{\hat P}^*)$ as representations of $(A,P)$.

Conversely, suppose that $\phi:A\rightarrow A^*$ is the linear
isomorphism giving the equivalence between $(A, L, R, P)$ and
$(A^*, R^*,L^*,Q^*)$. Define a bilinear form $\frakB(\;,\;)$ on
$A$ by
$$\frakB(a,b):=\langle \phi(a), b\rangle,\;\;\forall a,b\in A.$$
Then a similar argument gives the Rota-Baxter Frobenius algebra $(A,P,\frakB)$
and $\hat P=Q$.
\end{proof}

We now extend the notion of double constructions to Rota-Baxter Frobenius algebras.

 \begin{defi}   \mlabel{de:3.4}
 Let $(A, \cdot, P)$ be a Rota-Baxter algebra. Suppose that $(A^*,\circ, Q^*)$ is a Rota-Baxter algebra.
 A {\bf double construction of Rota-Baxter Frobenius algebra} associated to $(A, \cdot, P)$ and $(A^*,\circ, Q^*)$ is a double construction $(A\bowtie A^*, \frakB_d)$ of Frobenius algebra associated to $(A, \cdot)$ and $(A^*,\circ)$ such that $(A\bowtie A^*,P+ Q^*, \frakB_d)$ is a  Rota-Baxter Frobenius algebra, that is, $P+Q^*$ is a Rota-Baxter operator on $A\bowtie A^*$.

 \end{defi}

By definition, both $(A, \cdot, P)$ and $(A^*,\circ, Q^*)$ are
Rota-Baxter subalgebras of $(A\bowtie A^*,P+ Q^*)$. Thus the
Rota-Baxter algebra $(A\bowtie A^*,P+Q^*)$ comes from a matched
pair of the Rota-Baxter algebras $(A,P)$ and $(A^*,Q^*)$ in Theorem~\mref{thm:3.1}.

\begin{lem}
\mlabel{lem:abas}
Let $(A\bowtie A^*, P+Q^*, \mathfrak{B}_d)$ be a double construction of Rota-Baxter Frobenius algebra associated to $(A,P)$ and $(A^*,Q^*)$.
\begin{enumerate}
\item The adjoint $\widehat{ P+Q^*}$ of $P+Q^*$ with respect to $\frakB_d$ is $Q+ P^*$. Further $Q+P^*$ is \admt $(A\bowtie A^*, P+Q^*)$.
\mlabel{it:abas1}
\item $Q$ is admissible to $(A, P)$.
\mlabel{it:abas2}
\item $P^*$ is admissible to $(A^*, Q^*)$.
\mlabel{it:abas3}
\end{enumerate}
\end{lem}

\begin{proof}
(\mref{it:abas1}) For $a,b\in A, a^*,b^*\in
A^*$, by Eq.~(\mref{eq:3.9}), we have {\small
\begin{eqnarray*}
\frakB_d\big((P+Q^*)(a+a^*),b+b^*\big)&=&\frakB\big(P(a)+Q^*(a^*),b+b^*\big)
=\langle P(a),b^*\rangle + \langle Q^*(a^*), b\rangle\\
&=&\langle a, P^*(b^*)\rangle + \langle a^*, Q(b)\rangle
=\frakB_d(a+a^*, (Q+P^*)(b+b^*)).
\end{eqnarray*}
}
Hence the adjoint $\widehat{ P+Q^*}$ of $P+Q^*$ with respect to
$\frakB_d$ is $Q+ P^*$. By Proposition~\ref{pp:frobadm},
$\widehat{ P+Q^*}=Q+P^*$ is \admt $(A\bowtie
A^*, P+Q^*)$.

\smallskip

\noindent
 (\mref{it:abas2}) By Item~(\mref{it:abas1}), $Q+P^*$ is admissible to $(A\bowtie A^*, P+Q^*)$. By Eqs.~(\mref{eq:it:2.3a}) and (\mref{eq:it:2.3b}), this is true if and only if, for $a,b\in A, a^*,b^*\in A^*$,
 \begin{eqnarray*}
 &&(Q+P^*)((a+a^*)(P(b)+Q^*(b^*)))=(Q(a)+P^*(a^*))(P(b)+Q^*(b^*))\\
 &&+(Q+P^*)((Q(a)+P^*(a^*))(b+b^*))+\lam (Q(a)+P^*(a^*))(b+b^*)
 \end{eqnarray*}
 and
 \begin{eqnarray*}
 &&(Q+P^*)((P(a)+Q^*(a^*))(b+b^*))=(P(a)+Q^*(a^*))(Q(b)+P^*(b^*))\\
 &&+(Q+P^*)((a+a^*)(Q(b)+P^*(b^*)))+\lam (a+a^*)(Q(b)+P^*(b^*)).
 \end{eqnarray*}

Now taking $a^*=b^*=0$ in the above equations gives the admissibility of $Q$ to $(A,P)$.
\smallskip

\noindent
 (\mref{it:abas3}) Likewise, taking $a=b=0$ in the above equations yields the desired admissibility:
 \begin{eqnarray*}
 &&P^*(a^*Q^*(b^*))=P^*(a^*)Q^*(b^*)+P^*(P^*(a^*)b^*)+\lam P^*(a^*)b^*,\\
 &&P^*(Q^*(a^*)b^*)=Q^*(a^*)P^*(b^*)+P^*(a^*P^*(b^*))+\lam a^*P^*(b^*). \qquad \qquad \qquad \qquad \qquad \qedhere
 \end{eqnarray*}
\end{proof}

Extending Theorem~\mref{thm:frob} to Rota-Baxter Frobenius algebras, we obtain

\begin{thm} Let $(A,\cdot,P)$ be a Rota-Baxter algebra. Suppose that there is a Rota-Baxter algebra structure $(A^*,\circ,Q^*)$ on its dual space $A^\ast$. Then there is a double construction of Rota-Baxter Frobenius
algebra $(A\oplus A^*,P_{A\oplus A^*},\frakB_d)$ associated to $(A,\cdot,P)$ and $(A^*,\circ,Q^*)$ if and only if
$((A,P),(A^*,Q^*), R^*_\cdot, L^*_\cdot,R^*_\circ,L^*_\circ)$ is a
matched pair of Rota-Baxter algebras. \mlabel{thm:3.7}
\end{thm}

\begin{proof}
($\Longrightarrow$) The given double construction $(A\oplus
A^*,P_{A\oplus A^*},\frakB_d)$ associated to the
Rota-Baxter algebras $(A,P)$ and $(A^*,Q^*)$ implies that
$(A\oplus A^*,\frakB_d)$ is a double construction of Frobenius
algebra associated to $A$ and $A^*$. Hence by
Theorem~\mref{thm:frob}, $(A,A^*, {R_\cdot}^*, {L_\cdot}^*,
{R_\circ}^*, {L_\circ}^*)$ is a matched pair of algebras for which
the algebra on $A\oplus A^*$ is the algebra $A\bowtie A^*$. Since
the Rota-Baxter operator $P_{A\oplus A^*}$ is $P+Q^*$,
$((A,P),(A^*,Q^*), {R_\cdot}^*, {L_\cdot}^*, {R_\circ}^*,
{L_\circ}^*)$ is a matched pair of Rota-Baxter algebras.
\smallskip

\noindent ($\Longleftarrow$) If $((A, P), (A^*, Q^*), {R_\cdot}^*,
{L_\cdot}^*, {R_\circ}^*, {L_\circ}^*)$ is a matched pair of
Rota-Baxter algebras, then $(A,A^*, $ ${R_\cdot}^*, {L_\cdot}^*,
{R_\circ}^*, {L_\circ}^*)$ is a matched pair of algebras. Hence by
Theorem~\mref{thm:frob} again, $(A\bowtie A^*,\frakB_d)$ is a
Frobenius algebra. By Theorem~\mref{thm:3.1}, the matched pair of
Rota-Baxter algebras also equips the algebra
$A\bowtie A^*$ with the Rota-Baxter operator $P+Q^*$, giving us a
Rota-Baxter Frobenius algebra. This is what we need.
\end{proof}

Combining Theorems~\ref{thm:rbinfbialg1} and~\ref{thm:3.7}, we have
\begin{thm}
Let $(A, \cdot,P)$ be a Rota-Baxter algebra. Suppose that there is
a Rota-Baxter algebra $(A^*, \circ,Q^*)$ on the linear dual
$A^*$ of $A$. Let $\Delta:A\rightarrow A\otimes A$ denote the linear dual of
the multiplication $\circ:A^*\ot A^*\to A^*$ on $A^*$. Then the following conditions are equivalent.
 \begin{enumerate}
 \item The sextuple $((A,P),(A^*,Q^*), {R_\cdot}^*, {L_\cdot}^*, {R_\circ}^*, {L_\circ}^*)$ is a matched pair of Rota-Baxter algebras.
\mlabel{it:rbb1}
\item There is a double construction of
Rota-Baxter Frobenius algebra associated to $(A,\cdot,P)$ and
$(A^*,\circ,Q^*)$. \mlabel{it:rbb2}
\item The quintuple
$(A,\cdot,\Delta,P,Q)$ is a Rota-Baxter \asi bialgebra.
\mlabel{it:rbb3}
\end{enumerate}
\mlabel{thm:rbbial}
\mlabel{thm:rbinfbialg}
\end{thm}

\section{Coboundary Rota-Baxter ASI bialgebras, admissible associative Yang-Baxter equations and $\mathcal{O}$-operators}
\mlabel{sec:aybe}
In this section, we study the coboundary
Rota-Baxter \asi bialgebras and show that they can be given by antisymmetric
solutions of the \qadm associative Yang-Baxter equation in a Rota-Baxter algebra.
 We also give the notion of
$\mathcal{O}$-operators on a Rota-Baxter algebra and show that
$\mathcal{O}$-operators provide antisymmetric solutions of \qadm
associative Yang-Baxter equation in suitable Rota-Baxter algebras.

\subsection{Coboundary Rota-Baxter ASI bialgebras}
\mlabel{ss:cbdy}
For given $r\in A\otimes A$, define
 \begin{equation}
 \Delta(a):=\Delta_r(a):=(\id\otimes L(a)-R(a)\otimes \id)(r), \quad \forall a\in A.
 \mlabel{eq:4.1}
 \end{equation}

\begin{defi}  \mlabel{de:4.1}
A Rota-Baxter \asi bialgebra $((A,P), \Delta, Q)$ is called {\bf coboundary} if $\Delta$ is defined by Eq.~(\mref{eq:4.1}) for some $r\in A\ot A$.
 \end{defi}

 \begin{rmk} Let $(A, P)$ be a $Q$-admissible Rota-Baxter algebra and $r\in A\otimes A$. If $\Delta: A\to  A\otimes A$ is given by Eq.~(\mref{eq:4.1}), then $\Delta$ satisfies Eq.~(\mref{eq:3.14}). Moreover, by \cite[Proposition 2.3.4]{Bai1}, $\Delta$ satisfies Eq.~(\mref{eq:3.15}) if and only if $r$ satisfies
 \begin{equation}\mlabel{eq:4.9}
 (L(a)\otimes \id-\id\otimes R(a))(\id\otimes L(b)-R(b)\otimes \id)(r+\sigma(r))=0, \;\; \forall a, b\in A.
 \end{equation}
 By \cite[Proposition 2.3.3]{Bai1}, we know that $\Delta^*$ defines an algebra on $A^*$ if and only if
 \begin{equation}\mlabel{eq:4.2}
 (\id\otimes \id\otimes L(a)-R(a)\otimes \id\otimes \id)(r_{12}r_{13}+r_{13}r_{23}-r_{23}r_{12})=0, \;\;\forall a\in A,
 \end{equation}
where
for $r=\sum_ia_i\otimes b_i$,
\begin{equation*}
r_{12}r_{13}=\sum_{i,j}a_ia_j\otimes b_i\otimes
b_j,\;\;r_{13}r_{23}=\sum_{i,j}a_i\otimes a_j\otimes
b_ib_j,\;\;r_{23}r_{12}=\sum_{ij} a_j\otimes a_ib_j\otimes b_i.
\end{equation*}

Hence in order for $((A, P), \Delta, Q)$ to be a Rota-Baxter \asi
bialgebra, we just need to further require that $(A^*, \Delta^*,
Q^*)$ is a $P^*$-admissible Rota-Baxter algebra, that is,
$(A,\Delta,Q)$ is a Rota-Baxter coalgebra and Eqs.
(\mref{eq:pduqrd}) and (\mref{eq:pduqld}) hold.
  \mlabel{rmk:4.2}
 \end{rmk}

\begin{thm} \mlabel{thm:pq}
Let $(A, P)$ be a $Q$-admissible Rota-Baxter algebra and $r\in A\otimes A$. Define a linear map $\Delta: A\to  A\otimes A$ by Eq.~$($\mref{eq:4.1}$)$. Suppose that $\Delta^*$ defines an associative
multiplication on $A^*$. Then the following conclusions hold.
 \begin{enumerate}
\item \mlabel{it:pq1} Eq.~$($\mref{eq:corb}$)$ holds if and only
if for $a\in A$,
 \begin{eqnarray}\mlabel{eq:corbo}
 &&\ \ \big(\id\otimes Q( L(a))-\id\otimes L(Q(a))\big)(Q\otimes \id-\id\otimes P)(r)\\
 &&+\big(Q( R(a))\otimes \id-R(Q(a))\otimes \id\big)(P\otimes \id-\id\otimes Q)(r)=0.\nonumber
 \end{eqnarray}
  \item \mlabel{it:pq2}
 Eq.~$($\mref{eq:pduqrd}$)$ holds if and only if for  $a\in A$,
 \begin{eqnarray}\mlabel{eq:P*admissible1}
 &&\big(\id\otimes L(P(a))-R(P(a))\otimes \id+\id\otimes Q( L(a))\\
 &&+P( R(a))\otimes \id+\lambda \id\otimes L(a)\big)(P\otimes \id-\id\otimes Q)(r)=0.\nonumber
 \end{eqnarray}
\item \mlabel{it:pq3} Eq.~$($\mref{eq:pduqld}$)$ holds if and only
if for $a\in A$,
 \begin{eqnarray}\mlabel{eq:P*admissible2}
 &&\big(\id\otimes L(P(a))-R(P(a))\otimes \id-\id\otimes P( L(a))\\
 &&-Q( R(a))\otimes \id-\lambda R(a)\otimes \id\big)(Q\otimes \id-\id\otimes P)(r)=0.\nonumber
 \end{eqnarray}
 \end{enumerate}
 \end{thm}
\vspace{-.3cm}
\begin{proof} (\mref{it:pq1}). Set $r=\sum_i a_i\otimes b_i$. By Eq.~(\mref{eq:4.1}), we have
{\small
\begin{eqnarray*}
 &&(\id\otimes Q( L(a))-\id\otimes L(Q(a)))(Q\otimes \id-\id\otimes P)(r)\\
 &&+(Q( R(a))\otimes \id-R(Q(a))\otimes \id)(P\otimes \id-\id\otimes Q)(r)\\
&=& \sum_{i}Q(a_i)\otimes Q(ab_i)-Q(a_ia)\otimes Q(b_i)-Q(a_i)\otimes Q(a)b_i+Q(P(a_i)a)\otimes b_i\\
 &&-P(a_i)Q(a)\otimes b_i-a_i\otimes Q(aP(b_i))+a_i\otimes Q(a)P(b_i)+a_iQ(a)\otimes Q(b_i)\\
 &=&\sum_{i}Q(a_i)\otimes Q(ab_i)-Q(a_ia)\otimes Q(b_i)-Q(a_i)\otimes Q(a)b_i+Q(P(a_i)a)\otimes b_i\\
 &&-P(a_i)Q(a)\otimes b_i-\lambda a_iQ(a)\otimes b_i-a_i\otimes Q(aP(b_i))+a_i\otimes Q(a)P(b_i)\\
 &&+\lambda a_i\otimes Q(a)b_i+a_iQ(a)\otimes Q(b_i)-\lambda a_i\otimes Q(a)b_i+\lambda a_iQ(a)\otimes b_i\\
 &\stackrel{(\mref{eq:pduqr})(\mref{eq:pduql})}{=}&
 \sum_{i}Q(a_i)\otimes Q(ab_i)-Q(a_ia)\otimes Q(b_i)-Q(a_i)\otimes Q(a)b_i+Q(a_iQ(a))\otimes b_i\\
 &&-a_i\otimes Q(Q(a)b_i)+a_iQ(a)\otimes Q(b_i)-\lambda a_i\otimes Q(a)b_i+\lambda a_iQ(a)\otimes b_i\\
&\stackrel{(\mref{eq:4.1})}{=}& (Q\ot Q)\Delta(a)-(Q\ot \id+\id\ot
Q)\Delta Q(a)-\lambda \Delta Q(a).
\end{eqnarray*}}
Hence Eq.~$($\mref{eq:corb}$)$ holds if and only if
Eq.~(\ref{eq:corbo}) holds. Items~(\mref{it:pq2}) and
(\mref{it:pq3}) can be proved by the same argument.
 \end{proof}

By Remark~\mref{rmk:4.2} and Theorem~\mref{thm:pq}, we have
 \begin{cor} Let $(A, P)$ be a $Q$-admissible Rota-Baxter algebra and $r\in A\otimes A$. Then the linear map $\Delta$ defined by Eq.~$($\mref{eq:4.1}$)$ induces a $P^*$-admissible Rota-Baxter algebra
  $(A^*, \Delta^*, Q^*)$ such that $((A, P), \Delta, Q)$ is a Rota-Baxter \asi bialgebra if and only if
  Eqs. $($\mref{eq:4.9}$)$-$($\mref{eq:P*admissible2}$)$ are satisfied.
\mlabel{thm:4.6a}
 \end{cor}
\vspace{-.3cm}
\begin{ex}\label{ex:cobo}
We continue with the notations in Example~\ref{ex:directsum}. Suppose in addition that $(A,\cdot_A,\Delta_A)$ and $(B,\cdot_B,\Delta_B)$ are
coboundary ASI bialgebras, that is, there exist $r_1\in
A\otimes A$ and $r_2\in B\otimes B$ such that
\vspace{-.2cm}
\begin{eqnarray*}
\Delta_A(a)&=&(\id\otimes L(a)-R(a)\otimes \id)(r_1),\ \forall a\in A,\\
\Delta_B(b)&=&(\id\otimes L(b)-R(b)\otimes \id)(r_2),\ \forall b\in B.
\end{eqnarray*}
Then $(A\oplus B, \cdot_A+\cdot_B,\Delta_A+\Delta_B)$ is a
coboundary ASI bialgebra with
$r=r_1+r_2$:
\begin{equation}\label{eq:cob}
(\Delta_A+\Delta_B)(a+b)=(\id_{A\oplus B}\otimes
L(a+b)-R(a+b)\otimes \id_{A\oplus B})(r),\forall a\in A, b\in B.
\end{equation}
It is straightforward to check that Eqs.~(\ref{eq:corbo})-(\ref{eq:P*admissible2}) hold for
$P=P_A$, $Q=P_B$ and $r=r_1+r_2$. Hence
the Rota-Baxter ASI bialgebra $(A\oplus B,
\cdot_A+\cdot_B,\Delta_A+\Delta_B, P_A,P_B)$ is coboundary.
\end{ex}

We now prove a self-duality of a Rota-Baxter \asi bialgebra and give its construction on the double space.

 \begin{thm}  \mlabel{thm:4.7}
 Let $((A, P),\Delta, Q)$ be a Rota-Baxter \asi bialgebra. Let $\delta:A^*\to A^*\ot A^*$ be the linear dual of the multiplication on $A$. Then $((A^*,Q^*),-\delta,P^*)$ is also a Rota-Baxter \asi bialgebra.
 Further there is a Rota-Baxter \asi bialgebra structure on the direct sum $A\oplus A^*$ of
 the underlying vector spaces of $A$ and $A^*$ which contains the two Rota-Baxter \asi bialgebras as Rota-Baxter \asi sub-bialgebras.

 \end{thm}

 \begin{proof}
 Denote the product on the algebra $A^*$ by
$\circ$. By \cite[Remark 2.2.4]{Bai1}, $(A^*,\circ,-\delta)$ is an
ASI bialgebra. Moreover, $Q$ is admissible to the Rota-Baxter
algebra $(A,P)$ whose algebra structure is
given by $-\delta^*$ if and only if $Q$ is
admissible to the Rota-Baxter algebra $(A,P)$ whose algebra
structure is given by $\delta^*$. Therefore with the
fact that $P^*$ is admissible to $(A^*,Q^*)$, we show that
$((A^*,Q^*),-\delta,P^*)$ is a Rota-Baxter \asi bialgebra.

 Let $r\in A\otimes A^*\subset (A\oplus A^*)\otimes (A\oplus A^*)$ correspond to the identity map $\id$ on $A$. Let $\{e_1, e_2, \cdots, e_n\}$ be a basis of $A$ and $\{e^1, e^2, \cdots, e^n\}$ its dual basis.
 Then $r=\sum^n_{i=1}e_i\otimes e^i$.
 Let $(A\bowtie A^*, \star)$ denote the
 algebra structure on $A\oplus A^*$ induced by the matched pair $(A,A^*, {R_\cdot}^*, {L_\cdot}^*, {R_\circ}^*,
 {L_\circ}^*)$ of algebras. Define
 $$\Delta_{A\bowtie A^*}(u)=(\id\otimes L_{A\bowtie A^*}(u)-R_{A\bowtie A^*}(u)\otimes
 \id)(r),\;\;\forall u\in A\bowtie A^*.$$
Moreover, $(A\bowtie A^*, P+Q^*)$ is a $(Q+P^*)$-admissible
Rota-Baxter algebra by Lemma~\mref{lem:abas}. Hence
Eqs.~(\mref{eq:pduqr}) and (\mref{eq:pduql}) hold.
Since
\vspace{-.3cm}
\small{
 \begin{eqnarray*}
((P+Q^*)\otimes \id- \id\otimes (Q+ P^*))(r)
 &=&\sum^n_{i=1}(P(e_i)\otimes e^i-e_i\otimes
 P^*(e^i))\stackrel{}{=}0,\\
((Q+P^*)\otimes \id- \id\otimes (P+ Q^*))(r)
\vspace{-.3cm}
 &=&\sum^n_{i=1}(Q(e_i)\otimes e^i-e_i\otimes Q^*(e^i))\stackrel{}{=}0,
 \end{eqnarray*}
}
Eqs.~(\mref{eq:corbo})--(\mref{eq:P*admissible2}) hold. By
\emph{\cite[Theorem 2.3.6]{Bai1}}, we know that $r$ satisfies
Eqs.~(\mref{eq:4.9}) and (\mref{eq:4.2}), and  $(A\bowtie
A^*,\star, \Delta_{A\bowtie A^*})$ is an \asi bialgebra containing $(A,\cdot,\Delta)$ and
$(A^*,\circ,-\delta)$ as \asi sub-bialgebras. Therefore
$((A\bowtie A^*, P+Q^*), \star, \Delta_{A\bowtie A^*}, Q+P^*)$ is
a Rota-Baxter \asi bialgebra. It is obvious that it contains
$((A,P),\Delta,Q)$ and $((A^*,Q^*),-\delta,P^*)$ as Rota-Baxter
\asi sub-bialgebras. This completes the proof.
\end{proof}
%\vspace{-.3cm}
\subsection{Admissible associative Yang-Baxter equation in a Rota-Baxter algebra}
\mlabel{ss:ybe}
As a consequence of Corollary~\mref{thm:4.6a}, we obtain

 \begin{cor}   \mlabel{cor:4.10}
 Let $(A, P)$ be a $Q$-admissible Rota-Baxter algebra and $r\in A\otimes A$.  Then the linear map $\Delta$ defined by Eq.~\eqref{eq:4.1} induces a
 $P^*$-admissible Rota-Baxter algebra $(A^*, \Delta^*, Q^*)$ such that $((A, P), \Delta, Q)$ is a Rota-Baxter  \asi bialgebra
 if Eq.~$($\ref{eq:4.9}$)$ and the following equations hold:
\vspace{-.2cm}
\begin{eqnarray}\mlabel{eq:4.10}
 &r_{12}r_{13}+r_{13}r_{23}-r_{23}r_{12}=0,&
 \\
 \mlabel{eq:4.11}
& (P\otimes \id-\id\otimes Q)(r)=0,&
\\
\mlabel{eq:4.11a}
& (Q\otimes \id-\id\otimes P)(r)=0.&
\end{eqnarray}
\end{cor}
% \vspace{-.2cm}
This leads us to the following variation of the associative Yang-Baxter equation.
\begin{defi}
Let $(A, P)$ be a Rota-Baxter algebra. Suppose that $r\in A\otimes
A$ and $Q:A\rightarrow A$ is a linear map. Then
Eq.~(\mref{eq:4.10}) with conditions given by
Eqs.~(\mref{eq:4.11}) and (\mref{eq:4.11a}) is called the {\bf
\qadm associative Yang-Baxter equation (AYBE) in $(A, P)$} or
simply the {\bf \qadm \aybe}.
 \mlabel{de:4.11}
 \end{defi}

\begin{rmk}\mlabel{rmk:4.10a}
Eq.~(\mref{eq:4.10}) is simply the associative
Yang-Baxter equation (AYBE) in an associative algebra~\cite{Ag2,Bai1}, as an analogue of the classical Yang-Baxter equation in a Lie algebra~\cite{D}. Also if $r$ is antisymmetric
$($that is, $r=-\sigma(r)$$)$, then Eq.~(\mref{eq:4.11}) holds if and only if Eq.~(\mref{eq:4.11a}) holds.
\end{rmk}

\begin{rmk} Continuing with the notations in Example~\ref{ex:cobo},
we make an observation that distinguishes the admissible AYBE from the AYBE.
Suppose that $r_1$ and $r_2$ are distinct skew-symmetric solutions of the AYBE in the algebras $A$ and $B$ respectively. Then $r=r_1+r_2$ is a skew-symmetric solution of the AYBE in
$A\oplus B$. However, $r$ is not a solution of the $P_B$-admissible AYBE in
the Rota-Baxter algebra $(A\oplus B, P_A)$ since in this case
$(P_A\otimes \id_{A\oplus B}-\id_{A\oplus B}\otimes P_B)(r)=r_1-r_2$ is nonzero.
\end{rmk}

By Corollary~\mref{cor:4.10}, we have the following conclusion.

 \begin{cor}  \mlabel{rmk:4.12}
Let $(A, P)$ be a $Q$-admissible Rota-Baxter algebra and $r\in A\otimes A$ an antisymmetric solution of the \qadm \aybe in $(A, P)$. Then $((A, P), \Delta, Q)$ is  a Rota-Baxter \asi bialgebra,
 where the linear map $\Delta=\Delta_r$ is defined by Eq.~\meqref{eq:4.1}.
 \end{cor}

We now study solutions of the \qadm \aybe. For a vectors space
$A$, the isomorphism $A\ot A\cong \Hom(A^*,K)\ot A\cong
\Hom(A^*,A)$ identifies an $r\in A\otimes A$ with a map from
$A^*$ to $A$ which we still denote by $r$. Explicitly, writing
$r=\sum_{i}a_i\otimes b_i$, then
\vspace{-.2cm}
\begin{equation}\mlabel{eq:4.12}
r:A^*\to A, \quad
r(a^*)=\sum_{i}\langle a^*, a_i
\rangle b_i, ~~\forall a^*\in A^*. \vspace{-.4cm}
\end{equation}
 We called $r\in A\otimes A$ {\bf nondegenerate} if the map $r: A^*\to  A$ defined by Eq.~(\mref{eq:4.12}) is bijective.

 \begin{thm} \mlabel{thm:4.14}
 Let $(A, \cdot, P)$ be a Rota-Baxter algebra and $r\in A\otimes A$ antisymmetric. Let $Q:A\rightarrow A$ be a linear map. Then $r$ is a solution of the \qadm \aybe in $(A, \cdot, P)$ if and only if $r$
 satisfies
 \begin{equation}\mlabel{eq:4.161}
 r(a^*)\cdot r(b^*)=r({R_A}^*(r(a^*))b^*+a^*{L_A}^*(r(b^*))), ~~\forall a^*, b^*\in A^*,
 \end{equation}
 \begin{equation}\mlabel{eq:4.16}
 P r=r Q^*.
 \end{equation}
 \end{thm}
\vspace{-.3cm}
 \begin{proof} By \cite[Proposition 2.4.7]{Bai1}, $r$ is a solution of the \aybe in $(A,
 \cdot)$ if and only if Eq.~(\mref{eq:4.161}) holds. Moreover, let
 $r=\sum_ia_i\otimes b_i$ and for $a^*\in A^*$, we have
 $$
 r(Q^*(a^*))\stackrel{}{=}\sum_{i}\langle Q^*(a^*), a_i \rangle b_i= \sum_{i}\langle a^*, Q(a_i) \rangle b_i, \quad P(r(a^*))\stackrel{}{=}\sum_{i}\langle a^*, a_i \rangle P(b_i).
 $$
 So $Pr=r   Q^*$ if and only if Eq.~(\mref{eq:4.11}) holds. This completes the proof.   \end{proof}

We next relate the admissible \aybe to Connes cocycles and Frobenius algebras.
 \begin{defi} \mlabel{de:1.5}
 {\rm \mcite{Bai1}}  An antisymmetric bilinear form $\omega: A\otimes A\to  K$ on an algebra $A$ is a {\bf cyclic 1-cocycle} in the sense of Connes~\mcite{Co}, or simply a {\bf Connes cocycle}, if
 \begin{equation}\mlabel{eq:1.4}
 \omega(a b, c)+\omega(b c, a)+\omega(c a, b)=0,  ~~\forall~a, b, c\in A.
 \end{equation}
 \end{defi}

Let $(A, P)$ be a Rota-Baxter algebra and $\omega: A\otimes
A\to  K$ be a bilinear form. Suppose that $\omega$ is a
nondegenerate Connes cocycle on $A$. Define $\hat P:A\to  A$ to be the (right) adjoint linear transformation of $P$ with respect to $\omega$:
\begin{equation}
\omega(P(a), b)=\omega(a, \hat P(b)),\;\;\forall a,b\in A.
\mlabel{eq:adjoint-skew}
\end{equation}

 \begin{pro}Let $(A, P)$ be a Rota-Baxter algebra. Let
$r\in A\otimes A$ and $Q:A\rightarrow A$ be a linear map.
Suppose that $r$ is antisymmetric and nondegenerate. Let $\omega$ be the bilinear form defined by the inverse $r^{-1}:A\to A^*$ of the linear bijection corresponding to $r$:
$$\omega(a, b):=\langle r^{-1}(a), b \rangle, \quad \forall a, b\in A. $$
Then $r$ is
a solution of the \qadm \aybe in $(A, P)$ if and only if $\omega$ gives a nondegenerate Connes cocycle  $(A, \omega)$ with respect to which $Q$ is the adjoint $\hat P$  of $P$.
 \mlabel{pro:4.16}
 \end{pro}

 \begin{proof} By \emph{\cite[Proposition 2.1]{Ag2}}, we only need to check that Eq.~(\mref{eq:4.11}) holds if and only if the adjoint $\hat P$  of $P$ with respect to $\omega$ is $Q$. But the latter holds if and only if for all $a, b\in A$,
$$
0=\omega(P(a), b)-\omega(a, Q(b))= \langle r^{-1}  P(a), b \rangle -\langle r^{-1}(a), Q(b) \rangle
=\langle r^{-1}  P(a), b \rangle -\langle Q^*  r^{-1}(a), b \rangle,$$
which holds if and only if  $r^{-1}  P=Q^*  r^{-1}$, which means $P  r=r  Q^*$, that is, $(P\otimes \id-\id\otimes
 Q)(r)=0$.
This completes the proof.
 \end{proof}

Now let $(A, P, \mathfrak{B})$ be a  Rota-Baxter   symmetric
Frobenius algebra of weight zero. Then under the natural bijection
$\Hom(A\ot A,K)\cong \Hom(A,A^*)$, the bilinear form $\frakB$
corresponds to the linear map (see also the
proof of Proposition~\ref{pp:frobadm})
 $$
 \phi: A\to  A^*,~~\langle \phi(a), b \rangle:=\mathfrak{B}(a, b),~~\forall a, b\in A.
 $$
By pre-composing, we obtain a bijection
\begin{equation} \mlabel{eq:4.19}
 \Hom(A^*,A) \to \Hom(A,A), r\mapsto P_r:=r\phi, \quad \forall r\in \Hom(A^*,A).
 \end{equation}
\begin{thm}
Let $(A, P, \mathfrak{B})$ be a  Rota-Baxter   symmetric Frobenius algebra of weight zero. Then an antisymmetric $r\in A\otimes A$ is a solution of the $\hat P$-admissible \aybe in $(A, P)$ if and only if the corresponding $P_r$ from Eq.~\eqref{eq:4.19} is a Rota-Baxter operator of weight zero on $A$ such that $P  P_r=P_r  P$.
\mlabel{thm:4.23}
 \end{thm}

 \begin{proof} By \cite[Corollary 3.17]{BGN1}, $r$ is a solution of AYBE in
the algebra $A$, that is, Eq.~(\mref{eq:4.161}) holds, if and only
if $P_r$ is a Rota-Baxter operator of weight zero on $A$.

Moreover, set $r=\sum_ia_i\otimes b_i$. For $a\in A$, we have {\small
\begin{eqnarray*}
P  P_r(a)&=&P  r(\phi(a)), \\
P_r  P(a)&=&r(\phi(
P(a)))=\sum_{i}\mathfrak{B}(P(a),a_i)b_i =\sum_{i}\mathfrak{B}(a,
\hat P(a_i))b_i\\
&=&\sum_i\langle \phi(a), \hat P(a_i)\rangle b_i=r
\hat P^*(\phi(a)).
\end{eqnarray*}
}
\vspace{-.2cm}
Thus $P  r=r  \hat P^*$ if and only if $P  P_r=P_r  P$. \end{proof}

\begin{cor} Let $(A, P, \mathfrak{B})$ be a Rota-Baxter  symmetric  Frobenius algebra of weight zero and $r\in A\otimes A$ antisymmetric. If $(A, P_r)$ is a Rota-Baxter algebra of weight zero and $P  P_r=P_r  P$, where $P_r$ is defined by  Eq.~$($\mref{eq:4.19}$)$, then $((A,P), \Delta, \hat P)$ is a Rota-Baxter \asi  bialgebra, where $\Delta=\Delta_r$ is given by Eq.~$($\mref{eq:4.1}$)$.
In particular, if $\hat P=-P$, then for the inverse image $r=r_P$ of $P$ under the bijection in Eq.~\meqref{eq:4.19}, the triple $((A,P), \Delta_r, -P)$ is a Rota-Baxter \asi  bialgebra.
\mlabel{cor:4.25}
 \end{cor}

 \begin{proof} Note that $(A,P)$ is a $\hat P$-admissible Rota-Baxter algebra due to Proposition~\mref{pp:frobadm}. Hence the first conclusion follows from Corollary~\mref{rmk:4.12} and Theorem~\mref{thm:4.23}.

Now suppose $\hat P=-P$. Note that the element $r_P\in \Hom(A^*,A)$ corresponding to $P$ under the bijection in Eq.~\eqref{eq:4.19} is defined by
\begin{equation}
    r_P(a^*):=P(\phi^{-1}(a^*)),\;\;\forall a^*\in A^*.
\end{equation}
Then, for $a^*, b^*\in A^*$,
\begin{eqnarray*}
\langle r_P(a^*), b^*\rangle+\langle a^*,r_P(b^*)\rangle&=&
\langle P(\phi^{-1}(a^*)), b^*\rangle+\langle a^*,
P(\phi^{-1}(b^*))\rangle\\
&=&\frakB(P(\phi^{-1}(a^*)), \phi^{-1}(b^*))
+\frakB(\phi^{-1}(a^*),P(\phi^{-1}(b^*)))\\&=&\frakB(\phi^{-1}(a^*),
\hat
P(\phi^{-1}(b^*)))+\frakB(\phi^{-1}(a^*),P(\phi^{-1}(b^*)))=0.
\end{eqnarray*}
Hence $r_P$ (or its corresponding $2$-tensor) is antisymmetric. Since $P_{r_P}=P$ under the bijection in Eq.~\eqref{eq:4.19}, the second conclusion
holds from the first conclusion. 
\end{proof}

\subsection{$\mathcal{O}$-operators on Rota-Baxter algebras}
\mlabel{ss:oop}
The importance of Theorem \mref{thm:4.14} leads us to the next notion.
 \begin{defi}
 Let $(A, P)$ be a Rota-Baxter algebra of weight $\lambda$. Let $(V, \ell, r)$ be a representation of the algebra $A$ and $\alpha:V\to V$ be a linear map. A linear map $T: V\to  A$ is called a {\bf weak $\mathcal{O}$-operator associated to
 $(V, \ell, r)$ and $\alpha$} if $T$ satisfies
 \begin{equation}\mlabel{eq:4.18}
 T(u)T(v)=T(\ell(T(u))v+u r(T(v))), \quad \forall u, v\in V,
 \end{equation}
 \begin{equation}\mlabel{eq:4.17}
 P  T=T  \alpha.
 \end{equation}
 If in addition, $(V,\ell,r,\alpha)$ is a representation of
 $(A,P)$, then $T$ is called an {\bf $\mathcal{O}$-operator associated to
 $(V, \ell, r,\alpha)$.}
 \mlabel{de:4.17}
 \end{defi}

 \begin{ex}
 \begin{enumerate}
 \item
 Let $(A, P)$ be a Rota-Baxter algebra of weight $\lambda$. Then the identity map $\id$ on $A$ is an $\mathcal{O}$-operator associated to $(A, L, 0, P)$ or $(A, 0, R, P)$.
 \mlabel{ex:4.18}
 \item
 Let $(A, P)$ be a Rota-Baxter algebra of weight zero. Then the Rota-Baxter operator  $P: A\to  A$ is an $\mathcal{O}$-operator associated to the representation $(A, L, R, P)$.
\mlabel{ex:4.19}
\end{enumerate}
\mlabel{ex:oop}
\end{ex}

Theorem \mref{thm:4.14} is rewritten in terms of $\mathcal
O$-operators as follows.

\begin{cor}
Let $(A, P)$ be a Rota-Baxter algebra and $r\in A\otimes A$
antisymmetric. Let $Q:A\to A$ be a linear map.
 Then $r$ is a solution of the \qadm \aybe in $(A, P)$ if and only if $r$ is a weak $\mathcal{O}$-operator associated to $(A^*, R^*, L^*)$ and $Q^*$.
If in addition, $(A, P)$ is a $Q$-admissible Rota-Baxter algebra,
then  $r$ is a solution of the \qadm \aybe in
$(A, P)$ if and only if $r$ is an $\mathcal{O}$-operator
associated to the representation $(A^*, R^*, L^*, Q^*)$.
 \mlabel{ex:4.20}
 \end{cor}
We next show that $\mathcal O$-operators give numerous solutions of
the admissible \aybe in semi-direct product Rota-Baxter algebras and give rise to
Rota-Baxter \asi bialgebras.

We first consider admissible quadruples for semi-direct products
of Rota-Baxter algebras.

\begin{thm}
Let $(A,P)$ be a Rota-Baxter algebra of weight $\lambda$ and let
$(V,\ell,r)$ be a representation of the algebra $A$.
 Let $Q:A\to A$
and $\alpha,\beta:V\to V$ be linear maps. Then the following
conditions are equivalent.
\begin{enumerate}
\item There is a Rota-Baxter algebra
$(A\ltimes_{\ell,r}V,P+\alpha)$ such that
 the linear operator $Q+\beta$ on $A\oplus V$ is \admt
$(A\ltimes_{\ell,r}V,P+\alpha)$. \mlabel{it:admdual1} \item There
is a Rota-Baxter algebra $(A\ltimes_{r^*,\ell^*}V^*,P+\beta^*)$
such that the linear operator $Q+\alpha^*$ on $A\oplus V^*$ is
\admt $(A\ltimes_{r^*,\ell^*}V^*,P+\beta^*)$. \mlabel{it:admdual2}
\item The following conditions are satisfied:
\begin{enumerate}
\item \label{it:admdual3.1} $(V,\ell,r,\alpha)$ is a
representation of $(A,P)$, that is, Eqs.~$($\mref{eq:2.1}$)$ and
$($\mref{eq:2.2}$)$ hold; \item \label{it:admdual3.2} $Q$ is \admt
$(A,P)$, that is, Eqs.~$($\mref{eq:pduqr}$)$ and
$($\mref{eq:pduql}$)$ hold; \item \label{it:admdual3.3}
$(V,l,r,\beta)$ is an admissible quadruple of $(A,P)$, that is,
Eqs.~$($\mref{eq:it:2.3a}$)$ and $($\mref{eq:it:2.3b}$)$ hold;
\item \label{it:admdual3.4} For $a\in A,u\in
V$, we have
\begin{equation}\mlabel{eq:ab3-1}
\beta(\ell(a)\alpha(u))=\beta(\ell(Q(a))u)+\ell(Q(a))\alpha(u)+\lambda
\ell (Q(a))u,\end{equation}
\begin{equation}\mlabel{eq:ab3-2}
\beta(\alpha(u)r(a))=\beta(ur(Q(a)))+\alpha(u)r(Q(a))+\lambda
ur(Q(a)).
\end{equation}
\end{enumerate}
\mlabel{it:admdual3}
\end{enumerate}
\mlabel{thm:admdual}
\end{thm}

\begin{proof}
(\mref{it:admdual1}) $\Longleftrightarrow$
(\mref{it:admdual3})\quad By Proposition~\mref{pro:2.2},
$(A\ltimes_{\ell,r}V,P+\alpha)$ is a Rota-Baxter algebra if and
only if $(V,\ell,r,\alpha)$ is a representation of the Rota-Baxter
algebra $(A,P)$. Let the product on $(A\ltimes_{\ell,r}V)$ be
denoted by $*$. Let $a,b\in A$ and $u,v\in V$. Then we have
\begin{eqnarray*}
&&(Q+\beta)(((P+\alpha)(a+u))*(b+v))=Q(P(a)\cdot b)+\beta (\alpha
(u)r(b))+\beta(\ell(P(a))v),\\
&&((P+\alpha)(a+u))*((Q+\beta)(b+v))=P(a)\cdot
Q(b)+\alpha(u)r(Q(b))+\ell(P(a))\beta(v),\\
&&(Q+\beta)((a+u)*((Q+\beta)(b+v)))=Q(a\cdot
Q(b))+\beta(ur(Q(b)))+\beta(\ell(a)\beta(v)), \\
&&\lambda(a+u)*((Q+\beta)(b+v))=\lambda a\cdot Q(b)+\lambda
ur(Q(b))+\lambda \ell(a)\beta(v).
\end{eqnarray*}
Therefore Eq.~(\mref{eq:pduql}) holds (where $Q$ is replaced by
$Q+\beta$, $P$ by $P+\alpha$, $a$ by $a+u$, and $b$ by $b+v$) if
and only if Eq.~(\mref{eq:pduql}) (corresponding to $u=v=0$),
Eq.~(\mref{eq:it:2.3b}) (corresponding to $b=u=0$) and
Eq.~(\mref{eq:ab3-2}), where $a$ is replaced by $b$,
(corresponding to $a=v=0$) hold. Similarly, Eq.~(\mref{eq:pduqr})
holds (where $Q$ is replaced by $Q+\beta$, $P$ by $P+\alpha$, $a$
by $a+u$, and $b$ by $b+v$) if and only if Eq.~(\mref{eq:pduqr}),
Eq.~(\mref{eq:it:2.3a}) and Eq.~(\mref{eq:ab3-1})  hold. Hence
Condition~(\mref{it:admdual1}) holds if and only if
Condition~(\mref{it:admdual3}) holds.

(\mref{it:admdual2}) $\Longleftrightarrow$
(\mref{it:admdual3})\quad In Item~\meqref{it:admdual1}, take
$$V=V^*,\ell=r^*, r=\ell^*, \beta=\alpha^*,\alpha=\beta^*.$$
Then from the above equivalence between
Condition~(\mref{it:admdual1}) and Condition~(\mref{it:admdual3}),
we have Condition~(\mref{it:admdual2}) holds if and only if the conditions (\ref{it:admdual3.1})-(\ref{it:admdual3.3}) in
Condition~(\mref{it:admdual3}) as well as the following two
equations hold (for all $a\in A, u^*\in V^*$):
\begin{equation}\mlabel{eq:ab3-1-1}
\alpha^*(r^*(a)\beta^*(u^*))=\alpha^*(r^*(Q(a))u^*)+r^*(Q(a))\beta^*(u^*)+\lambda
r^* (Q(a))u^*,\end{equation}
\begin{equation}\mlabel{eq:ab3-2-1}
\alpha^*(\beta^*(u^*)\ell^*(a))=\alpha^*(u^*\ell^*(Q(a)))+\beta^*(u^*)\ell^*(Q(a))+\lambda
u^*\ell^*(Q(a)).
\end{equation}
For $a\in A$, $u\in V$ and $u^*\in V^*$, we have
\begin{eqnarray*}
&&\langle \alpha^*(\beta^*(u^*)\ell^*(a)),u\rangle=\langle
\beta^*(u^*)\ell^*(a),\alpha(u))\rangle=\langle u^*,
\beta(\ell(a)\alpha(u))\rangle,\\
&&\langle \beta^*(u^*)\ell^*(Q(a)),u\rangle=\langle \beta^*(u^*),
\ell(Q(a))u\rangle=\langle u^*, \beta(\ell(Q(a))u)\rangle,\\
&&\langle \alpha^*(u^*\ell^*(Q(a))), u\rangle=\langle
u^*\ell^*(Q(a)), \alpha(u)\rangle=\langle u^*,
\ell(Q(a))\alpha(u)\rangle,\\
&&\langle \lambda u^*\ell^*(Q(a)),u\rangle=\langle u^*,\lambda
\ell(Q(a))u\rangle.
\end{eqnarray*}
Hence Eq.~(\mref{eq:ab3-2-1}) holds if and only if
Eq.~(\mref{eq:ab3-1}) holds. Similarly, Eq.~(\mref{eq:ab3-1-1})
holds if and only if Eq.~(\mref{eq:ab3-2}) holds. Therefore
Condition~(\mref{it:admdual2}) holds if and only if
Condition~(\mref{it:admdual3}) holds.
\end{proof}

Here is our main result on antisymmetric solutions of the admissible AYBE and the constructions of Rota-Baxter \asi bialgebras.
\begin{thm} \mlabel{thm:4.21}
Let $(V, \ell, r, \beta)$ be an admissible quadruple of a
Rota-Baxter algebra $(A, P)$ of weight $\lambda$ and let $(V^*,
r^*, \ell^*, \beta^*$) be the representation of $(A,P)$ defined in
Lemma \mref{lem:admrep}. Let $Q:A\to A$ and $\alpha:V\to V$ be linear maps.  Let $T: V\to  A$ be a linear map which is
identified as an element in $(A\ltimes_{r^*,\ell^*} V^*)\otimes
(A\ltimes_{r^*,\ell^*} V^*)$ $($through $\Hom(V,A)\cong V^*\ot A \subseteq (A\ltimes_{r^*,\ell^*} V^*)\otimes
(A\ltimes_{r^*,\ell^*} V^*)$$)$.
\begin{enumerate}
\item \mlabel{it:4.21a} The element $r=T-\sigma(T)$ is an
antisymmetric solution of the $(Q+\alpha^*)$-admissible \aybe in
the Rota-Baxter algebra $(A\ltimes_{r^*,\ell^*} V^*, P+ \beta^*)$
if and only if $T$ is a weak $\mathcal{O}$-operator associated to
$(V, \ell, r)$ and $ \alpha$, and $T  \beta=Q  T$. \item
\mlabel{it:4.21b} Assume that $(V,\ell,r,\alpha)$ is a representation  of
$(A,P)$. If $T$ is an ${\mathcal O}$-operator associated to
$(V,\ell,r,\alpha)$ and $T \beta=Q T$, then  $r=T-\sigma(T)$ is an
antisymmetric solution of the $(Q+\alpha^*)$-admissible \aybe in
the Rota-Baxter algebra $(A\ltimes_{r^*,\ell^*} V^*, P+\beta^*)$.
If, in addition, $(A,P)$ is $Q$-admissible and
Eqs.~$($\mref{eq:ab3-1}$)$-$($\mref{eq:ab3-2}$)$ hold, then the
Rota-Baxter algebra $(A\ltimes_{r^*,\ell^*}V^*,P+\beta^*)$ is
$(Q+\alpha^*)$-admissible. In this case, there is a Rota-Baxter
\asi bialgebra $((A\ltimes_{r^*,\ell^*} V^*, P+ \beta^*),\Delta,
Q+\alpha^*)$, where the linear map $\Delta=\Delta_r$ is defined by
Eq.~(\mref{eq:4.1}) with $r=T-\sigma(T)$.
\end{enumerate}
 \end{thm}

\begin{proof}
(\mref{it:4.21a}) \cite[Corollary 3.10]{BGN1} shows that $r$ satisfies Eq.~(\mref{eq:4.10}) if and only if Eq.~(\mref{eq:4.18}) holds.

Let $\{e_1, e_2,\cdots, e_n\}$ be a basis of $V$ and $\{e^1, e^2, \cdots, e^n\}$ be its dual basis. Then
$T=\sum_{i=1}^n T(e_i)\otimes e^i\in  (A\ltimes_{r^*,\ell^*} V^*)\otimes (A\ltimes_{r^*,\ell^*} V^*)$.
Hence
$$r=T-\sigma (T)=\sum_{i=1}^n (T(e_i)\otimes e^i-e^i\otimes T(e_i)).$$
Note that
\begin{eqnarray*}
((P+\beta^*)\otimes {\rm id})(r)&=&\sum_{i=1}^n(P  T(e_i)\otimes e^i-\beta^*(e^i)\otimes T(e_i)),\\
({\rm id}\otimes (Q+\alpha^*))(r)&=&\sum_{i=1}^n(T(e_i)\otimes \alpha^*(e^i)-e^i\otimes Q  T(e_i)).
\end{eqnarray*}
Further,
\begin{eqnarray*}
\sum_{i=1}^n \beta^*(e^i)\otimes T(e_i)&=&\sum_{i=1}^n\sum_{j=1}^n\langle \beta^*(e^i),e_j\rangle e^j\otimes T(e_i)= \sum_{j=1}^n e^j\otimes \sum_{i=1}^n\langle e^i,\beta(e_j)\rangle T(e_i)\\
&=&\sum_{i=1}^n e^i\otimes T(\sum_{j=1}^n\langle \beta(e_i),
e^j\rangle e_j)=\sum_{i=1}^n e^i\otimes T  \beta(e_i),
\end{eqnarray*}
and similarly,
$\sum_{i=1}^n T(e_i)\otimes \alpha^*(e^i)=\sum_{i=1}^n T  \alpha (e_i)\otimes e^i.
$
Therefore $((P+\beta^*)\otimes {\rm id})(r)=({\rm id}\otimes (Q+\alpha^*))(r)$ if and only if
$P  T=T  \alpha$ and $Q  T=T  \beta.$
Hence the conclusion follows.
\smallskip

\noindent
(\mref{it:4.21b}) It follows from Item~(\mref{it:4.21a}) and Theorem~\mref{thm:admdual}.
\end{proof}

\subsection{Some cases and examples}
\mlabel{ss:exam}
According to Theorem~\mref{thm:4.21},
from an $\mathcal O$-operator $T$ of a Rota-Baxter algebra $(A,P)$
associated to a representation $(V,l,r,\alpha)$ and an admissible
quadruple $(V,l,r,\beta)$ satisfying $T\beta=QT$, one can get a
skew-symmetric solution of the admissible AYBE in a semi-direct
product Rota-Baxter algebra. If in addition, $(A,P)$ is
$Q$-admissible and
Eqs.~$($\mref{eq:ab3-1}$)$-$($\mref{eq:ab3-2}$)$ hold, then one
can get a Rota-Baxter ASI bialgebra in the semi-direct product
Rota-Baxter algebra.

In general, $\beta$ is not related to $\alpha$. When $\beta$ does
depend on $\alpha$ in certain way: $\beta=\Pi(\alpha)$, say for a
Laurent series $\Pi\in K[x,x^{-1}]$, then it is natural to expect
that the double dual of a representation is the representation itself:
$\Pi(\beta)=\Pi^2(\alpha)=\alpha$. This happens when $\Pi(x)$ is
either $\pm x$, or $-x+\theta$, or $\theta x^{-1}$ when  $x$ is
invertible and $0\neq \theta\in K$, that is, when $\beta=\pm
\alpha$ or $-\alpha+\theta \id$ or $\beta =\theta \alpha^{-1}$. We
will investigate these instances more carefully since they provide
interesting examples and applications of our general construction
of Rota-Baxter ASI bialgebras.

To emphasize,  for all $\Pi$ in the set
$$ \admset,
\quad K^\times :=K\backslash \{0\},$$
we have $\Pi^2(\alpha)=\alpha$ and $\Pi(\alpha^*)=\Pi(\alpha)^*$.
Moreover, for any linear map $T:V\to A$, it is obvious that $T\Pi(\alpha)=\Pi(P)T$ when $T\alpha=PT$.

Applying Theorem~\mref{thm:admdual}, we conclude

\begin{pro}\label{co:semisemi}
Let $(A,P)$ be a Rota-Baxter algebra of weight $\lambda$. Let $(V,
\ell, r)$ be a representation of the algebra $A$ and $\alpha:V\to
V$ be a linear map. For $\Pi\in \admset$, there is a Rota-Baxter
algebra $(A\ltimes_{r^*,\ell^*}V^*,P+\Pi(\alpha^*))$ that
is $\Pi(P+\Pi(\alpha^*))$-admissible $($that is,
$(\Pi(P)+\alpha^*)$-admissible$)$ if and only if the {\bf
$\Pi$-admissible equations} $($associated to the quadruple
$(V,\ell,r,\alpha)$$)$ hold. Here
\begin{enumerate}
\item \mlabel{it:semisemi1}
when $\Pi=\theta x$ with $\theta=\pm 1$, the $\Pi$-admissible equations are
{\small
\begin{eqnarray}
&& \big(\theta P+P+\lambda \id\big) (aP(b))+\lambda P(ab)=0,\mlabel{eq:semi1}\\
 && \big(\theta P+P+\lambda \id\big)(P(a)b)+\lambda P(ab)=0,\mlabel{eq:semi2}\\
&&
\ell(P(a))\alpha(v)=\alpha(\ell(P(a))v)+\alpha(\ell(a)\alpha(v))+\lambda
\alpha(\ell(a)v)), \mlabel{eq:semi3}\\
&&\alpha(v)r(P(a))=\alpha(\alpha(v)r(a))+\alpha(v r(P(a)))+\lambda
\alpha(v r(a)), \mlabel{eq:semi4}\\
&&\big(\theta \alpha+\alpha+\lambda\id\big) (\ell(a)\alpha(v))+\lambda
\alpha(\ell(a)v) = 0, \mlabel{eq:semi5}\\
&&\big(\theta \alpha+\alpha+\lambda\id\big) (\alpha(v)r(a))+\lambda \alpha(v r(a))=0, \mlabel{eq:semi6}\\
 &&
 \big(\theta \alpha+\alpha+\lambda\id\big) (\ell(P(a))v)+\lambda\alpha(\ell(a)v)=0,\mlabel{eq:semi7}\\
 &&\big(\theta \alpha+\alpha+\lambda\id\big) (vr(P(a)))+\lambda \alpha (vr(a))=0, \quad \forall a,b\in A, v\in V;\mlabel{eq:semi8}
\end{eqnarray}
}
\item\label{it:semisemi1.5}
when $\Pi=-x+\theta$ with $\theta\ne 0$, the $\Pi$-admissible equations are Eqs.~\eqref{eq:semi3}-\eqref{eq:semi4} and
\begin{eqnarray}
&&(\lambda+\theta)(P(ab)+aP(b)-\theta a b)=0,\label{eq:semi1.5-1}\\
&&(\lambda+\theta)(P(ab)+P(a)b-\theta a b)=0,\label{eq:semi1.5-2}\\
&&(\lambda+\theta)(l(a)\alpha(v)+\alpha(l(a)v)-\theta l(a) v)=0,\label{eq:semi1.5-3}\\
&&(\lambda+\theta)(\alpha(v)r(a)+\alpha(vr(a))-\theta vr(a))=0,\label{eq:semi1.5-4}\\
&&(\lambda+\theta)(l(P(a))v+\alpha(l(a)v)-\theta l(a) v)=0,\label{eq:semi1.5-5}\\
&&(\lambda+\theta)(vr(P(a))+\alpha(vr(a))-\theta vr(a))=0,\quad \forall a,b\in A, v\in V;\label{eq:semi1.5-6}
\end{eqnarray}
\item \label{it:semisemi2} when $\Pi=\theta x^{-1}, \theta\neq 0$
$($in which case assume that $P$ and $\alpha$
are invertible$)$, the $\Pi$-admissible equations are
\begin{eqnarray}
&P(aP(b))=P(P(a)b)=\theta ab,&\mlabel{eq:semi-1}\\
& \alpha(\ell(a)\alpha(v))=\alpha(\ell(P(a))v)=\theta \ell(a)v,
\quad \alpha(\alpha(v)r(a))=\alpha(vr(P(a)))=\theta vr(a), &\mlabel{eq:semi-3}\\
&\hspace{-.5cm}
\ell(P(a))\alpha(v)=\big(\lambda \alpha+2\theta\id\big)(\ell(a)v),
\alpha(v)r(P(a))=\big(\lambda \alpha +2\theta\id\big)(vr(a)), \forall a, b\in A, v\in V.&
\mlabel{eq:semi-5}
\end{eqnarray}
\end{enumerate}
\end{pro}
\begin{proof} The statement is a consequence of Theorem~\mref{thm:admdual} in the case that
$\beta=\Pi(\alpha)$ and $Q=\Pi(P)$. We provide some details for Case~(\mref{it:semisemi1}). Cases~(\ref{it:semisemi1.5}) and (\mref{it:semisemi2}) are similarly verified.
\smallskip

\noindent
(\mref{it:semisemi1})
 Let $\beta=\theta\alpha$ and $Q=\theta P$. Then by Theorem~\mref{thm:admdual},
$(A\ltimes_{r^*,\ell^*}V^*,P+\theta \alpha^*)$ is a Rota-Baxter
algebra that is $(\theta P+\alpha^*)$-admissible if and only if the
following conditions hold.
\begin{enumerate}
\item[(i)] $(V,\ell,r,\theta\alpha)$ is an admissible quadruple of $(A,P)$.
 By Lemma~\mref{lem:admrep}, $(V^*, r^*, \ell^*, \theta\alpha^*)$ is a representation of
 $(A, P)$ if and only if
 \begin{eqnarray}
 & \alpha(vr(P(a)))-\alpha(v)r(P(a))-\theta \alpha(\alpha(v)r(a))-\lambda\alpha(v)r(a)=0,& \mlabel{eq:it:2.3aa} \\
 & \alpha(\ell(P(a))v)-\ell(P(a))\alpha(v)-\theta \alpha(\ell(a)\alpha(v))-\lambda \ell(a)\alpha(v)=0,\  \forall v\in V, a\in A.&
 \mlabel{eq:it:2.3bb}
 \end{eqnarray}
\item[(ii)] $\theta P$ is admissible to $(A,P)$. Applying Item~(i)
to the adjoint representation
$(V,L,R,P)$, this holds if and only if Eqs.~(\mref{eq:semi1}) and
(\mref{eq:semi2}) hold. \item[(iii)] $(V,\ell,r,\alpha)$ is a
representation of $(A,P)$. It holds if and only if
Eqs.~(\mref{eq:semi3}) and (\mref{eq:semi4}) hold.
\item[(iv)]
Eqs.~(\mref{eq:ab3-1}) and (\mref{eq:ab3-2}) hold, where
$\beta=\theta\alpha$ and $Q=\theta P$.
\end{enumerate}
A straightforward computation shows that Eqs.~(\ref{eq:it:2.3aa}),
(\ref{eq:it:2.3bb}), (\mref{eq:semi3}),(\mref{eq:semi4}),
(\mref{eq:ab3-1}) with $\beta=\theta\alpha$ and $Q=\theta P$ and
Eq.~(\mref{eq:ab3-2}) with $\beta=\theta\alpha$ and $Q=\theta P$ hold
if and only if Eqs.~(\mref{eq:semi3})-(\mref{eq:semi8}) hold.
\end{proof}

Taking $\lambda=0$ and $\Pi=-x$ in
Proposition~\mref{co:semisemi}.(\mref{it:semisemi1}),
and $\theta=-\lambda$ with $\lambda\ne 0$ and hence $\Pi=-x-\lambda$ in Proposition~\mref{co:semisemi}.(\mref{it:semisemi1.5}),
 we obtain an important
case where the admissibility does not impose any
restraints.

\begin{cor}
\mlabel{co:special}
Let $(A,P)$ be a Rota-Baxter algebra of weight $\lambda$
and $(V, \ell, r, \alpha)$ be a representation of $(A, P)$. Then $(V^*,r^*,\ell^*,-\alpha^*-\lambda\id_{V^*})$
is a representation of $(A,P)$ and hence
gives rise to a Rota-Baxter algebra $(A\ltimes_{r^*,\ell^*}V^*,P-
\alpha^*-\lambda\id_{V^*})$. Moreover,
the linear operator $-P-\lambda\id_A+\alpha^*$ on $A\oplus V^*$ is \admt
$(A\ltimes_{r^*,\ell^*}V^*,P-\alpha^*-\lambda\id_{V^*})$.
\end{cor}

For antisymmetric solutions of the admissible AYBE and the resulting Rota-Baxter \asi bialgebras, we have
\begin{pro}
 Let $(A, P)$ be a Rota-Baxter algebra of weight $\lambda$. Let $(V, \ell, r)$ be a representation of
the algebra $A$.  Let $\alpha:V\to V$ and $T: V\to  A$ be linear
maps. Let $\Pi\in \admset$.
\begin{enumerate}
\item Let $(V, \ell, r, \Pi(\alpha))$ be an admissible quadruple
of $(A,P)$. Then $r=T-\sigma(T)$ is an antisymmetric solution of
the $(\Pi(P)+\alpha^*)$-admissible \aybe
in $(A\ltimes_{r^*,\ell^*} V^*, P+\Pi(\alpha^*))$ if and only if
$T$ is a weak $\mathcal{O}$-operator associated to $(V, \ell, r)$
and $\alpha$. \mlabel{it:cons1} \item Assume the validity of the
$\Pi$-admissible equations, given respectively by
Eqs.~$($\mref{eq:semi1}$)$-$($\mref{eq:semi8}$)$ for $\Pi=\pm x$,
by Eqs.~\eqref{eq:semi3}-\eqref{eq:semi4}, \eqref{eq:semi1.5-1}-\eqref{eq:semi1.5-6} for
$\Pi\in -x+K^\times$,
and by Eqs.~\meqref{eq:semi-1}-\meqref{eq:semi-5} for $\Pi\in
K^{\times}x^{-1}$. Then $(V,\ell,r,\alpha)$ is a representation of
$(A,P)$ and there is a $(\Pi(P)+\alpha^*)$-admissible Rota-Baxter
algebra $(A\ltimes_{r^*,\ell^*}V^*,P+\Pi(\alpha^*))$. If $T$ is an
${\mathcal O}$-operator associated to $(V,\ell,r,\alpha)$, then
$r=T-\sigma(T)$ is an antisymmetric solution of the
$(\Pi(P)+\alpha^*)$-admissible \aybe in the Rota-Baxter algebra
$(A\ltimes_{r^*,\ell^*} V^*, P +\Pi(\alpha^*))$. Further there is
a Rota-Baxter \asi bialgebra $((A\ltimes_{r^*,\ell^*} V^*,
P+\Pi(\alpha^*)),\Delta,\Pi(P)+\alpha^*)$, where the linear map
$\Delta=\Delta_r$ is defined by Eq.~(\mref{eq:4.1}) with
$r=T-\sigma(T)$. \mlabel{it:cons2}
\end{enumerate}
\mlabel{cor:cons}
\end{pro}

\begin{proof}
(\mref{it:cons1}) follows from Theorem~\mref{thm:4.21}.(\mref{it:4.21a}).

\smallskip

\noindent
(\mref{it:cons2}) follows from Proposition~\mref{co:semisemi} and Theorem~\mref{thm:4.21}.(\mref{it:4.21b}).
\end{proof}

We next focus on the case when $Q=-P-\lambda \id$ for a Rota-Baxter algebra $(A,P)$ of weight $\lambda$. In this case, by Corollary~\ref{co:special}, the $\Pi$-admissible equations associated to $(V,l,r,\alpha)$
hold automatically. Hence Proposition~\ref{cor:cons}.(\ref{it:cons2}) gives

\begin{cor}
Let $(A, P)$ be a Rota-Baxter algebra of weight $\lambda$.
Let $T:V\rightarrow A$ be an
${\mathcal O}$-operator  associated to a representation $(V,\ell,r,\alpha)$ of $(A,P)$. Then
$r=T-\sigma(T)$ is an antisymmetric solution of the
$(-P-\lambda\id_A+\alpha^*)$-admissible \aybe in the Rota-Baxter algebra
$(A\ltimes_{r^*,\ell^*} V^*, P -\alpha^*-\lambda\id_{V^*})$. Further there is
a Rota-Baxter \asi bialgebra $((A\ltimes_{r^*,\ell^*} V^*,
P-\alpha^*-\lambda\id_{V^*}),\Delta,-P-\lambda\id_A+\alpha^*)$, where the linear map
$\Delta=\Delta_r$ is defined by Eq.~(\mref{eq:4.1}) with
$r=T-\sigma(T)$.
\mlabel{cor:cons2}
\end{cor}

We then display explicitly solutions of the admissible \aybe that give Rota-Baxter \asi bialgebras obtained from the $\mathcal O$-operators
in Example~\ref{ex:oop}.

\begin{cor}
Let $(A, P)$ be a Rota-Baxter algebra of weight $\lambda$.
\begin{enumerate}
\item Denote $r_1:=\sum^n_{i=1}
(e_i\otimes e^i-e^i\otimes e_i)$, where $\{e_1,\cdots, e_n\}$ is a
basis of $A$ and $\{e^1,\cdots, e^n\}$ is its dual basis. Then $r_1$ is a
solution of the $(-P-\lambda\id_A+P^*)$-admissible \aybe in the Rota-Baxter algebras $(A\ltimes_{R^*, 0} A^*, P-P^*-\lambda\id_{A^*})$ and
$(A\ltimes_{0,L^*} A^*, P-P^*-\lambda\id_{A^*})$.
 Moreover, the two Rota-Baxter algebras are $(-P-\lambda
 \id_A+P^*)$-admissible and hence there are
 Rota-Baxter \asi bialgebras $((A\ltimes_{R^*, 0} A^*, P-P^*-\lambda\id_{A^*}),\Delta,-P-\lambda\id_A+P^*)$ and $((A\ltimes_{0, L^*} A^*, P-P^*-\lambda\id_{A^*}),\Delta,-P-\lambda\id_A+P^*)$, where the linear map $\Delta=\Delta_{r_1}$ is defined
by Eq.~$($\mref{eq:4.1}$)$ with the above $r_1$. \mlabel{it:4.23a}
\item Suppose that $\lambda=0$.
The element $r_2:=P-\sigma(P)$ is a solution of the
$(-P+P^*)$-admissible \aybe in the Rota-Baxter algebra
$(A\ltimes_{R^*, L^*} A^*, P-P^*)$.
 Moreover, there
 is a Rota-Baxter \asi bialgebra $((A\ltimes_{R^*, L^*} A^*,
P-P^*),\Delta,-P+P^*)$, where the linear map
$\Delta=\Delta_{r_2}$ is defined by Eq.~$($\mref{eq:4.1}$)$ with the above $r_2$. \mlabel{it:4.23aa}
\end{enumerate}

\mlabel{cor:4.23}
\end{cor}
\begin{proof} (\mref{it:4.23a})  By Example~\mref{ex:oop}.(\mref{ex:4.18}), $\id$ is an $\mathcal{O}$-operator associated to $(A, L, 0, P)$ or $(A, 0, R,
P)$. Note that ${\rm id}=\sum^n_{i=1} (e_i\otimes e^i)$. Hence the
conclusion follows from
Corollary~\mref{cor:cons2}.

\smallskip

\noindent (\mref{it:4.23aa}) By
Example~\mref{ex:oop}.(\mref{ex:4.19}), the Rota-Baxter operator
$P$ is an $\mathcal{O}$-operator associated to $(A, L, R, P)$.
Hence the conclusion follows from
Corollary~\mref{cor:cons2}.
\end{proof}

\section{Dendriform algebras and Rota-Baxter dendriform algebras}
\mlabel{sec:dend}

We first recall the well-known fact that a Rota-Baxter algebra induces a dendriform
algebra. Then we show that a Rota-Baxter \asi bialgebra gives a quadri-bialgebra introduced in \mcite{NB}. We also prove that a Rota-Baxter
dendriform algebra gives an $\mathcal{O}$-operator on the
associated Rota-Baxter algebra and hence induces a Rota-Baxter
\asi bialgebra.

\subsection{Dendriform algebras and quadri-bialgebras}
\mlabel{ss:dend}
\begin{defi} \mcite{Lo} Let $A$ be a vector space with multiplications $\prec$ and $\succ$. Then $(A, \prec, \succ)$ is called a {\bf dendriform algebra} if for all $a, b, c\in A$,
\small{
$$
 (a\prec b)\prec z=a\prec (b\prec c+b\succ c),
 (a\succ b)\prec c=a\succ (b\prec c),
 (a\prec b+a\succ b)\succ c=a\succ (b\succ c).
$$
}
\end{defi}

 Let $(A,\prec,\succ)$ be a dendriform algebra.
For $a\in A$, let $L_{\prec}(a)$, $R_{\prec}(a)$
and $L_{\succ}(a)$, $R_{\succ}(a)$ denote the left and right
multiplication operators on $(A,\prec)$ and $(A,\succ)$,
respectively.
 Furthermore, define linear maps
 $$R_{\prec}, L_{\succ}: A\to  \End_K(A), \quad
 a\mapsto R_{\prec}(a),~a\mapsto L_{\succ}(a),\;\;\forall a\in A.
 $$

As is well known, for a dendriform algebra $(A,\prec,\succ)$, the multiplication
\begin{equation} a*b=a\prec b+a\succ b,\quad \forall a,b\in A\mlabel{eq:5.41}\end{equation}
defines an algebra $(A,*)$, called the {\bf associated algebra} of
the dendriform algebra. Moreover, $(A,L_\succ, R_\prec)$ is a
representation of the algebra  $(A,*)$~\mcite{Bai1,Lo}. Further, a
Rota-Baxter operator $P$ of weight zero on an algebra $(A,\cdot)$
induces a dendriform algebra~\mcite{Ag}:
\begin{equation}
a\succ b=P(a)\cdot b, a\prec b=a\cdot P(b), \quad
\forall a,b\in A.\mlabel{eq:Ag}\end{equation}

It is natural to ask what algebraic structures can be  induced
from the pair of Rota-Baxter operators on the algebra and the
coalgebra in a Rota-Baxter \asi bialgebra of weight
zero. To address this question, we first recall the following
notions.

\begin{defi} {\rm (\mcite{Bai1,NB})}  Let $(A,\prec,\succ)$ be a dendriform
algebra and let $(A, *)$  be the associated algebra.
A symmetric bilinear form $\frakB:A\otimes A\rightarrow K$ is called a {\bf $2$-cocycle} of
$(A,\prec,\succ)$ if $\frakB$ satisfies
\begin{equation}
\frakB(a* b,c)=\frakB(b,c\prec a)+\frakB(a,b\succ c), \quad \forall
a,b,c\in A.\end{equation}
A {\bf Manin triple of
dendriform algebras} with respect to a nondegenerate 2-cocycle is a
triple of dendriform algebras $(A,A^{+},A^{-})$ together with a
nondegenerate $2$-cocycle ${\frakB}$ on $A$, such that
\begin{enumerate}
\item $A^{+}$ and $A^{-}$ are dendriform subalgebras of $A$; \item
$A=A^{+}\oplus A^{-}$ as vector spaces; \item $A^{+}$ and $A^{-}$
are isotropic with respect to $\frakB$.
\end{enumerate}
\end{defi}

\begin{lem}
Let $(A, P, \mathfrak{B})$ be a Rota-Baxter symmetric
Frobenius algebra of weight zero. Then $\mathfrak{B}$ satisfies
\begin{equation}
    \mathfrak{B}(a\prec b,c)=\mathfrak{B}(a,b\succ c),\;\;\forall
    a,b,c\in A,
\end{equation}
where $(A,\prec,\succ)$ is the dendriform algebra defined in
Eq.~(\mref{eq:Ag}). In particular, $\mathfrak{B}$ is a  2-cocycle
of the dendriform algebra $(A,\prec,\succ)$.
\mlabel{lem:quadri}
\end{lem}
\begin{proof} Let $\cdot$ denote the associative product on $(A,P)$. Let $a,b,c\in A$. Then we have
    \begin{eqnarray*}
        \frakB(a\prec b,c)=\frak B(a\cdot P(b),c)=\frak B(a,P(b)\cdot
        c)=\frak B(a,b\succ c).
    \end{eqnarray*}
    Furthermore, we have
    \begin{eqnarray*}
        \frakB(a*b,c)=\frakB(a\prec b+a\succ b,c)=\frak B(a,b\succ
        c)+\frak B(c,a\succ b)=\frak B(a,b\succ c)+\frak B(b,c\prec a).
    \end{eqnarray*}
    Hence $\mathfrak{B}$ is a 2-cocycle of the dendriform algebra
    $(A,\prec,\succ)$. 
\end{proof}

\begin{cor} Let $((A, P), \Delta, Q)$ be a Rota-Baxter
\asi bialgebra, where $(A, \cdot, P)$ is a
$Q$-admissible Rota-Baxter algebra of weight zero and  $(A^*,
\circ, Q^*)$ is a $P^*$-admissible Rota-Baxter algebra of weight zero for which the product $\circ$ is given by
$\Delta^*:A^*\otimes A^*\rightarrow A^*$. With the product $\star$ on
$A\bowtie A^*$ in Eq.~\eqref{eq:3.7}, define
$$x\succ y:=(P+Q^{*})(x)\star y,\;\;
x\prec y:=x\star (P+Q^{*})(y),\;\;\forall x,y\in A\oplus A^*.$$
Then $(A\oplus A^*, \prec,\succ)$ is a dendriform algebra and $(A,
\prec_{|_A},\succ_{|_A})$ ,
$(A^*,\prec_{|_{A^*}},\succ_{|_{A^*}})$ are dendriform
subalgebras. Moreover, with these dendriform algebra structures,
$(A\oplus A^*, A, A^*)$ is a Manin triple of dendriform algebras
with respect to the nondegenerate 2-cocycle $\mathfrak{B}_d$ defined
by Eq.~$($\mref{eq:3.9}$)$. \mlabel{cor:4.26}
 \end{cor}

 \begin{proof} By Theorem~\mref{thm:rbinfbialg}, $(A\bowtie A^*, \star, P+Q^{*})$ is a Rota-Baxter algebra of weight zero. Then by
Eq.~(\mref{eq:Ag}),  $(A\oplus A^*, \prec,\succ)$ is a dendriform
algebra and $(A, \prec_{|_A},\succ_{|_A})$ ,
$(A^*,\prec_{|_{A^*}},\succ_{|_{A^*}})$ are dendriform
subalgebras. Note that in fact $(A, \prec_{|_A},\succ_{|_A})$ ,
$(A^*,\prec_{|_{A^*}},\succ_{|_{A^*}})$ are the dendriform
algebras induced by the Rota-Baxter operators $P$ and $Q^*$
respectively by Eq.~(\mref{eq:Ag}).
Moreover, by Lemma~\mref{lem:quadri},
${\mathfrak{B}}_d$ is a symmetric nondegenerate 2-cocycle of $(A\oplus A^*, \prec, \succ)$. Hence the conclusion holds.
\end{proof}

\begin{rmk}\mlabel{rmk:4.26a}
There are several bialgebra theories for dendriform
algebras, such as dendriform bialgebra~\mcite{A4,LR1,LR2,Ron},
bidendriform bialgebra~\mcite{F2} and dendriform
$D$-bialgebra~\mcite{Bai1}. It came as a surprise that, by
\cite[Theorem 5.3]{NB}, a Manin triple of dendriform algebras
with respect to a nondegenerate 2-cocycle is none of them, but a
 {\bf quadri-bialgebra
} consisting of a quadri-algebra~\mcite{AL} and a quadri-coalgebra
satisfying certain compatibility conditions. See \cite[Definition
5.2]{NB}.
\end{rmk}

Then as a direct consequence of Corollary~\mref{cor:4.26}, we obtain
\begin{cor} \mlabel{co:quad}
For a Rota-Baxter \asi bialgebra $((A, P), \Delta, Q)$ of weight zero, the
Rota-Baxter operators $P$ and $Q^*$ induce a quadri-bialgebra.
\end{cor}

\subsection{Rota-Baxter dendriform algebras}
\mlabel{ss:rbdend} We introduce the notion of
a Rota-Baxter dendriform algebra and study its
relationships with $\mathcal{O}$-operators on Rota-Baxter
algebras.

 \begin{defi} Let $(A,\prec,\succ)$ be a  dendriform algebra. A linear operator $P$ on $A$ is called a {\bf Rota-Baxter operator of weight $\lambda$} if $P$ satisfies
\begin{equation}
 P(a)\diamond P(b)=P(P(a)\diamond b)+P(a\diamond P(b))+\lambda P(a\diamond b), \quad \forall a, b\in A, \diamond\in \{\prec, \succ\}.
 \mlabel{eq:5.5}
 \end{equation}
Then $(A, \prec, \succ, P)$ is called a {\bf Rota-Baxter dendriform algebra of weight $\lambda$}.
\mlabel{de:5.1}
 \end{defi}

By a simple verification, we obtain

 \begin{pro}\mlabel{rmk:5.2} Let $(A, \prec, \succ, P)$ be a Rota-Baxter dendriform algebra of weight $\lambda$. Then with the product $*$ given by Eq.~\meqref{eq:5.41}, $(A, *, P)$ is a Rota-Baxter algebra of weight $\lambda$, called
 the {\bf associated Rota-Baxter algebra of $(A,\prec,\succ,P)$}.
On the other hand, let $(A, \bullet, P)$ be a Rota-Baxter algebra of weight
zero, then $(A, \prec_\bullet, \succ_\bullet, P)$ is a Rota-Baxter dendriform
algebra of weight zero, where $\succ_\bullet,\prec_\bullet$ is given by Eq.~$($\mref{eq:Ag}$)$.
 \end{pro}

From Eqs.~(\mref{eq:2.1}), (\ref{eq:2.2}) and Definition~\mref{de:5.1}, we directly have

 \begin{pro}\mlabel{pro:5.3}Let $(A, \prec, \succ, P)$ be a Rota-Baxter dendriform algebra of weight $\lambda$. Then $(A, L_{\succ}, R_{\prec}, P)$ is a representation of the Rota-Baxter algebra $(A, \ast, P)$. Furthermore, the identity map $\id$ on $A$ is an $\mathcal O$-operator on the associated Rota-Baxter algebra $(A, \ast, P)$ associated to the representation $(A, L_{\succ}, R_{\prec}, P)$.
 \end{pro}

Therefore, by Proposition~\ref{cor:cons}, any Rota-Baxter dendriform algebra $(A,\prec,\succ,P)$ of weight $\lambda$ that
satisfies the corresponding $\Pi$-admissible equations associated to $(A, L_{\succ}, R_{\prec}, P)$ for  $\Pi\in \admset$ can give a solution of the admissible
AYBE induced from the identity map $\id$ and hence gives a Rota-Baxter ASI bialgebra.
We illustrate the construction explicitly by considering the case when $Q=-P-\lambda \id$, that is, the construction of Rota-Baxter \asi bialgebras by Rota-Baxter dendriform algebras of weight $\lambda$.

\begin{pro}\mlabel{cor:5.8}
Let $(A, \prec, \succ, P)$ be a Rota-Baxter dendriform
 algebra of weight $\lambda$. Let $(A, *, P)$ be the associated Rota-Baxter algebra of weight $\lambda$, where $*$ is given by Eq.~$($\mref{eq:5.41}$)$.
Let $\{e_1,\cdots, e_n\}$ be a basis of $A$ and $\{e^1,\cdots,
e^n\}$ be its dual basis. Denote
\vspace{-.2cm}
$$r:=\sum^n_{i=1} (e_i\otimes e^i-e^i\otimes e_i)
\vspace{-.2cm}
$$
and define the linear map $\Delta=\Delta_r$
by Eq.~$($\mref{eq:4.1}$)$.
Then $r$ is a solution
of the $(-P-\lambda\id_A+P^*)$-admissible \aybe in the Rota-Baxter algebra
$(A\ltimes_{R_\prec^*, L_\succ^*} A^*, P-P^*-\lambda\id_{A^*})$.
Further the Rota-Baxter algebra $(A\ltimes_{R_\prec^*, L_\succ^*} A^*, P-P^*-\lambda\id_{A^*})$ is $(-P-\lambda\id_A+P^*)$-admissible and hence there is a Rota-Baxter \asi bialgebra $((A\ltimes_{R_\prec^*, L_\succ^*} A^*,
 P-P^*-\lambda\id_{A^*}),\Delta,-P-\lambda\id_A+P^*)$.
 \end{pro}
 \begin{proof}
By Proposition~\mref{pro:5.3},
$\id$ is an $\mathcal{O}$-operator on the associated Rota-Baxter algebra $(A,\ast , P)$
associated to the representation $(A, L_{\succ}, R_{\prec}, P)$.
Note that ${\rm id}=\sum^n_{i=1} (e_i\otimes e^i)$.
Then the conclusion follows from Corollary~\mref{cor:cons2}.
\end{proof}

\begin{rmk}\mlabel{rmk:5.9}
 Corollary~ \mref{cor:4.23}.(\mref{it:4.23a}) is a special case of Proposition \mref{cor:5.8}, since the former corresponds to the trivial Rota-Baxter dendriform algebra structure on a Rota-Baxter algebra $(A, \cdot, P)$ given by $\succ=\cdot,\prec=0$ or $\succ=0, \prec=\cdot$.
 \end{rmk}

By Proposition~\mref{cor:5.8} we have

\begin{cor}\mlabel{cor:construction}
Let $(A,\cdot,P)$ be a Rota-Baxter algebra of weight zero. Let
$(A,\prec_\cdot,\succ_\cdot,P)$ be the Rota-Baxter dendriform
algebra of weight zero given in Proposition~ \mref{rmk:5.2}. Then
we have
$$
L_{\succ_\cdot}=L_\cdot P,\quad R_{\prec_\cdot}=R_\cdot P.
%\vspace{-.1cm}
$$
Let $(A,*_\cdot,P)$ be the associated Rota-Baxter algebra of
$(A,\prec_\cdot,\succ_\cdot,P)$. Then $(A,L_\cdot P, R_\cdot P,
P)$ is a representation of the Rota-Baxter algebra
$(A,*_\cdot,P)$. Let $r$ and $\Delta=\Delta_r$ be as defined in
Proposition~\mref{cor:5.8} and Eq.~$($\mref{eq:4.1}$)$
respectively. Then $r$ is a solution of the
$(-P+P^*)$-admissible \aybe in the Rota-Baxter algebra
$(A\ltimes_{(R_\cdot P)^*, (L_\cdot P)^*} A^*, P-P^*)$.
 Moreover,  this Rota-Baxter algebra is $(-P+P^*)$-admissible and hence there is a Rota-Baxter \asi bialgebra $((A\ltimes_{(R_\cdot P)^*, (L_\cdot P)^*}
A^*, P-P^*),\Delta,-P+P^*))$.
\end{cor}
\begin{rmk}
\mlabel{rk:final}
Combining Corollary~\mref{cor:4.23} and
Corollary~\mref{cor:construction}, from a
Rota-Baxter algebra $(A,\cdot,P)$  of weight zero, there are four
Rota-Baxter \asi bialgebras on the direct sum $A\oplus
A^*$ of the underlying vector spaces of $A$ and
$A^*$: $((A\ltimes_{R^*, L^*} A^*, P- P^*),\Delta,-P+P^*)$,
$((A\ltimes_{R^*, 0} A^*, P-P^*),\Delta,- P+P^*)$, $((A\ltimes_{0,
L^*} A^*, P- P^*),\Delta,-P+P^*)$ and $((A\ltimes_{(R_\cdot P)^*,
(L_\cdot P)^*} A^*, P- P^*),\Delta,- P+P^*)$, where the linear map
$\Delta$ is defined by Eq.~(\mref{eq:4.1}) with $r=P-\sigma(P)$
for the first case and with $r=\sum^n_{i=1} (e_i\otimes
e^i-e^i\otimes e_i) $ for the other three cases. Note that the
Rota-Baxter algebra structure on $A$ is $(A, \cdot, P)$ itself for
the first three cases and is $(A,*_\cdot,P)$ for the fourth case.
\end{rmk}

\noindent
 {\bf Acknowledgments.}  This work is supported by
 National Natural Science Foundation of China (Grant No.  11931009).  C. Bai is also
supported by the Fundamental Research Funds for the Central Universities and Nankai ZhiDe Foundation.

\end{document}